\documentclass[11pt]{article}
\usepackage{amsmath,amsthm,amsfonts,amssymb,amscd, amsxtra, mathrsfs,enumitem, mathtools}
\usepackage{url}
\usepackage[margin=2.3 cm,nohead]{geometry}
\usepackage{color}
\usepackage{pdfsync}
\usepackage{hyperref}
\usepackage{graphicx}
\usepackage{subcaption}
\usepackage{caption}
\usepackage{booktabs}    % for \toprule, \midrule, \bottomrule
\usepackage{float}  
\usepackage[
  backend=biber,
  style=ieee,
  citestyle=numeric-comp, % << this is the key
  url=false,
  isbn=false,
  sortcites=true
]{biblatex}
\usepackage{multicol}
\usepackage{multirow}
\usepackage{siunitx}
\setlist[enumerate]{label=(\roman*), align=left}
\newcommand{\matM}{\mathbb M}
\newtheorem{theorem}{Theorem}
\newtheorem{lemma}[theorem]{Lemma}

\newtheorem{proposition}[theorem]{Proposition}

\newtheorem{remark}{Remark}

\newtheorem{assumption}{Assumption}

\usepackage{graphicx,float}
\usepackage{graphicx}
\usepackage{grffile} % permite nomes de arquivos com espaços

\usepackage{algorithm}
\usepackage{algpseudocode}

\newcommand{\R}{\mathbb{R}}
\newcommand{\setS}{\mathcal{S}}

\newcommand{\inner}[2]{\langle #1, #2 \rangle}
  
\DeclareMathOperator{\diag}{diag}  
\DeclareMathOperator{\argmin}{argmin}  

\usepackage{amssymb}
\usepackage{relsize}
\algnewcommand{\Input}[1]{%
  \State \textbf{Input:} {\raggedright #1}
}

\algnewcommand{\Initialize}[1]{%
  \State \textbf{Initialize:}
  \Statex \hspace*{\algorithmicindent}\parbox[t]{.8\linewidth}{\raggedright #1}
}

\algnewcommand{\Output}[1]{%
  \State \textbf{Output:} {\raggedright #1}
}

\begin{document}
%%%%%%%%%%%%%%%%%%%%%%%%%%%%%%%%%%%%%%%%%%%%%%%%%%%%%%%%%%%%%
\title{State-Dependent Metric Projection Neural Network for Variational Inequalities}

\author{
Mohammed Alshahrani \thanks{Department of Mathematics, King Fahd University of Petroleum \& Minerals, Dhahran, 31261, Saudi Arabia\\
Interdisciplinary Research Center for Smart Mobility and Logistics, King Fahd University of Petroleum \& Minerals, Dhahran, 31261, Saudi Arabia, (e-mail:{\tt mshahrani@kfupm.edu.sa}).}
}

\maketitle

\begin{abstract}
Projection-based dynamical systems and projection neural networks offer a continuous-time approach to solving variational inequalities by driving the state toward its projection onto the feasible set. However, most existing models are built on a fixed Euclidean or constant metric, which can lead to poor conditioning and slow convergence when the operator or feasible set geometry is highly anisotropic. This paper introduces a state-dependent metric projection dynamical system, referred to as a state-dependent scaled projection neural network (SD-SPNN), in which the geometry of the projection operator evolves smoothly with the state. The dynamics generalize classical projected dynamical systems and projection neural networks by embedding continuous-time preconditioning directly into the flow. Under standard monotonicity assumptions on the operator and mild regularity conditions on the metric, we establish existence of solutions, an exact equilibrium–solution correspondence with the underlying variational inequality, and Lyapunov-based stability properties. The framework unifies Euclidean, constant-metric, and state-dependent projection neural network within a single continuous-time model. Numerical experiments illustrate how state-dependent metrics reshape the transient geometry of the projected dynamics while converging to the same equilibrium.

\end{abstract}


\medskip

% - 37N35 — Dynamical systems in optimization and economics  
% - 47H05 — Monotone operators and generalizations  
% - 49J40 — Variational inequalities  
% - 65K10 — Numerical optimization and variational techniques  
% - 65K15 — Numerical methods for variational inequalities and related problems  
% - 68T07 — Artificial neural networks and deep learning  
% - 90C25 — Convex programming  
% - 90C33 — Complementarity and equilibrium problems and variational inequalities

%%%%%%%%%%%%%%%%%%%%%%%%%%%%%%%%%%%%%%%%%%%%%%%%%%%%%%%%%%%%%%%%%%%%%%%
%%%%%%%%%%%%%%%%%%%%%%%%%%%%%%%%%%%%%%%%%%%%%%%%%%%%%%%%%%%%%%%%%%%%%%%
\section{Introduction}

Variational inequality (VI) problems provide a unifying framework for a wide
range of equilibrium models arising in optimization, game theory, network
analysis, and constrained dynamical systems. Given a closed convex set
$\mathcal S \subset \mathbb R^n$ and a mapping $F:\mathbb R^n \to \mathbb R^n$, the
classical variational inequality problem $\mathrm{VI}(F,\mathcal S)$ consists of
finding a point $\bar{x} \in \mathcal S$ such that
\[
\langle F(\bar{x}), x - \bar{x} \rangle \ge 0, \quad \forall x \in \mathcal S.
\]
Projection-based methods and their continuous-time counterparts have played a
central role in the numerical and dynamical treatment of such problems,
particularly under monotonicity-type assumptions on $F$.

A notable class of continuous-time approaches is given by projection dynamical
systems and projection neural networks, where the evolution of the state is
driven by the discrepancy between the current iterate and its projection onto
the feasible set. These models admit a clear equilibrium interpretation in terms
of VI solutions and allow Lyapunov-based stability analysis. Most existing
projection neural network, however, are formulated with respect to a fixed Euclidean
metric. While this choice leads to conceptually simple models, it may result in
poor conditioning and slow convergence when the operator $F$ is ill-conditioned
or when the geometry of the feasible set is poorly aligned with the Euclidean
structure. In discrete-time optimization, variable-metric and scaled projection methods are a natural remedy: by adapting the metric, they act as problem-aware preconditioners that improve conditioning and convergence behavior. This raises a continuous-time question: can one design projection neural network whose metric varies smoothly with the state, while retaining a clear equilibrium interpretation and stability guarantees?

In numerical optimization, variable-metric and scaled projection methods are
well established as effective tools for improving convergence behavior. By
introducing an appropriately chosen positive definite scaling matrix, such
methods exploit problem structure and mitigate ill-conditioning. This naturally
raises the question of whether similar ideas can be incorporated into
continuous-time projection neural network for variational inequalities. Allowing the
metric underlying the projection operator to depend on the current state,
however, fundamentally alters the structure of the resulting dynamical system:
the projection mapping itself becomes state dependent, and standard Lyapunov
arguments must account for the time variation of the metric along trajectories.

The objective of this paper is to initiate a systematic study of projection
dynamical systems for variational inequalities endowed with a state-dependent
metric. We propose a class of continuous-time dynamics based on metric
projections induced by a smoothly varying positive definite matrix and
investigate their well-posedness and stability properties. The analysis
explicitly accounts for the interaction between the projection operator and the
state-dependent metric and identifies conditions under which stability is
preserved despite this additional coupling.

The present work addresses this question by developing and analyzing a class of continuous-time projection neural network endowed with a state-dependent metric, and by identifying conditions under which their trajectories remain well-posed and stable.

\paragraph{Contributions.}
The main contributions of this paper are as follows:
\begin{itemize}
\item We introduce a state-dependent metric projection neural network for
variational inequalities, generalizing classical projection neural network based on
the Euclidean metric.
\item We establish an equilibrium characterization showing that solutions of
$\mathrm{VI}(F,\mathcal S)$ coincide with stationary points of the proposed
state-dependent projection neural network.
\item We develop a Lyapunov-based stability analysis that guarantees
well-posedness and asymptotic, and under stronger assumptions exponential,
convergence, under explicit regularity and boundedness conditions on the metric.
\item We present numerical experiments illustrating how constant and
state-dependent metric choices influence convergence behavior, particularly in
ill-conditioned settings.
\end{itemize}

At a conceptual level, the proposed framework may be viewed as a continuous-time
analogue of variable-metric projection methods, while simultaneously extending
projection neural networks beyond fixed-metric formulations. It provides a
flexible analytical setting in which metric design, stability theory, and
projection neural network can be studied in a unified manner. Detailed comparisons
with existing projection neural networks and scaled projection methods, as well
as a discussion of algorithmic implications, are deferred to later sections.

The remainder of the paper is organized as follows. Section~\ref{sec:notation} introduces
notation, definitions, and standing assumptions. Section~\ref{sec:well-posed} establishes
well-posedness and characterizes equilibria of the proposed dynamics. Sections~\ref{sec:stabI},\ref{sec:stabII} and \ref{sec:stabIII} are devoted to stability and convergence analysis. Section~\ref{sec:example} gives an illustrative example. Numerical experiments are presented in Section~\ref{sec:numerical}. Section~\ref{sec:relatedwork} is dedicated to the literature review and the related work, followed by concluding remarks and perspectives for further research in Section~\ref{sec:futurework}.

\section{Formulation and Notation}\label{sec:notation}

Let $\setS \subset \R^n$ be a nonempty closed convex set and let
$F : \R^n \to \R^n$ be a continuous mapping. The variational inequality
problem associated with $(F,\setS)$, denoted by $\mathrm{VI}(F,\setS)$,
consists of finding a point $\bar x \in \setS$ such that
\begin{equation}\label{eq:VI}
    \inner{F(\bar x)}{x - \bar x} \ge 0,
    \quad \forall x \in \setS.
\end{equation}

Throughout the paper, $\inner{\cdot}{\cdot}$ denotes the standard
Euclidean inner product on $\R^n$, with associated norm
$\|x\| = \sqrt{\inner{x}{x}}$.
Given a symmetric positive definite matrix $M \in \R^{n \times n}$,
we define the $M$--induced inner product and norm by
\[
\inner{x}{y}_M := x^T M y,
\qquad
\|x\|_M := \sqrt{\inner{x}{x}_M}.
\]

We consider a state-dependent metric generated by a mapping
$\matM : \R^n \to \R^{n \times n}$ satisfying
\[
\left[\matM(x)\right]^T = \matM(x),
\qquad
y^T \matM(x) y > 0 \quad \forall y \neq 0,
\]
for all $x \in \R^n$.
The corresponding inner product $\inner{\cdot}{\cdot}_{\matM(x)}$
and norm $\|\cdot\|_{\matM(x)}$ are referred to as
\emph{state-dependent}. 
\subsection*{Metric projection}

For a fixed $x \in \R^n$, we define the projection onto $\setS$
with respect to the metric induced by $\matM(x)$ as
\begin{equation}\label{eq:metric_projection}
    P_{\setS,\matM(x)}(y)
    :=
    \argmin_{z \in \setS}
    \frac{1}{2}\|z - y\|^2_{\matM(x)}.
\end{equation}
Since $\matM(x)$ is symmetric positive definite and $\setS$ is closed
and convex, the minimizer in \eqref{eq:metric_projection} exists and
is unique. Note that if $\matM(x)\equiv I$ (identity), then \eqref{eq:metric_projection} is the Euclidean projection. Whereas \eqref{eq:metric_projection} is called the scaled projection when $\matM(x)\equiv M$ (constant), see for example \cite{bonettini_scaled_2009}.  

The metric projection admits the following variational
characterization, which we state it as a lemma and include the proof for completeness. 
\begin{lemma}
For given $x,y \in \R^n$,  let $z \in \setS$. Then
\begin{equation}\label{eq:proj_char}
    z = P_{\setS,\matM(x)}(y)
    \quad \Longleftrightarrow \quad
    (z - y)^T \matM(x)(w - z) \ge 0,
    \quad \forall w \in \setS.
\end{equation}
\end{lemma}
\begin{proof}
Let $f:=P_{\setS,\matM(x)}$ then
\[
f(u)=\tfrac12 (u-y)^\top \matM(x)(u-y).
\]
Then $f$ is convex and $\nabla f(u)=\matM(x)(u-y)$. If $z=P_{\mathcal S,M}(y)$, then for any $w\in\mathcal S$ and $t\in[0,1]$,
$z+t(w-z)\in\mathcal S$. Define $\varphi(t)=f(z+t(w-z))$. Since $t=0$ is a
minimizer of $\varphi$,
\[
0\le \varphi'(0^+)=\nabla f(z)^\top (w-z)=(z-y)^\top M(w-z).
\]

On the other hand, if $(z-y)^\top M(w-z)\ge 0$ for all $w\in\mathcal S$. Then for any $w\in\mathcal S$,
\[
\begin{aligned}
f(w)-f(z)
&=\tfrac12 (w-y)^\top M(w-y)-\tfrac12 (z-y)^\top M(z-y)\\
&=\tfrac12 (w-z)^\top M(w-z)+(z-y)^\top M(w-z)\ge 0.
\end{aligned}
\]
Hence $z$ minimizes $f$ over $\mathcal S$, i.e., $z=P_{\mathcal S,M}(y)$.
\end{proof}

\subsection*{Standing assumptions}

The analysis relies on the following assumptions, which will be
maintained throughout the paper.

\begin{assumption}[Uniform metric boundedness]\label{ass:metric_bounds}
There exists a constant $L > 1$ such that
\[
\|\matM(x)\| \le L
\quad \text{and} \quad
\|\matM(x)^{-1}\| \le L,
\qquad \forall x \in \setS.
\]
\end{assumption}

Assumption~\ref{ass:metric_bounds} implies that the eigenvalues of
$\matM(x)$ lie in $[L^{-1},L]$ uniformly on $\setS$, and hence all norms
$\|\cdot\|_{\matM(x)}$ are uniformly equivalent to the Euclidean norm. In fact, we have 
\[
\frac{1}{\sqrt{L}}\|z\|\le \|z\|_{\matM(x)}\le \sqrt{L} \|z\|, \quad \forall~z\in \R^n.
\]

In addition, when stability properties are investigated, further
regularity of the metric will be required.

\begin{assumption}[Metric regularity]\label{ass:metric_regular}
The mapping $\matM(\cdot)$ is locally Lipschitz continuous on $\setS$.
\end{assumption}

\subsection*{State-dependent projection neural network}

Motivated by the fixed-metric projection neural network commonly used for
solving variational inequalities, we consider the following
state-dependent projection neural network:
\begin{equation}\label{eq:SD_dynamics}
    \frac{dx}{dt}
    =
    \lambda
    \Bigl[
    P_{\setS,\matM^{-1}(x)}
    \bigl(x - \alpha \matM(x) F(x)\bigr)
    - x
    \Bigr],
\end{equation}
where $\lambda > 0$ and $\alpha > 0$ are given parameters. When $\matM(x)=I$, system \eqref{eq:SD_dynamics} is precisely the classical projected neural network.

A point $\bar x \in \setS$ is an equilibrium of
\eqref{eq:SD_dynamics} if and only if
\begin{equation}\label{eq:equilibrium}
    \bar x
    =
    P_{\setS,\matM^{-1}(\bar x)}
    \bigl(\bar x - \alpha \matM(\bar x) F(\bar x)\bigr).
\end{equation}
Using the characterization \eqref{eq:proj_char}, it follows that
$\bar x$ satisfies \eqref{eq:equilibrium} if and only if it solves
$\mathrm{VI}(F,\setS)$.

The remainder of the paper is devoted to establishing well-posedness
and stability properties of the dynamical system
\eqref{eq:SD_dynamics} under appropriate assumptions on $F$ and
$\matM(\cdot)$.

\section{Well-posedness and Equilibria}\label{sec:well-posed}

In this section, we establish basic properties of the
state-dependent projection dynamical system
\eqref{eq:SD_dynamics}, including existence and uniqueness
of trajectories and a precise characterization of its equilibria.

\subsection{Existence and uniqueness of trajectories}

Define the vector field
\[
\Phi(x)
:=
P_{\setS,\matM^{-1}(x)}
\bigl(x - \alpha \matM(x)F(x)\bigr) - x,
\qquad x \in \setS.
\]
Then the dynamics \eqref{eq:SD_dynamics} can be written compactly as
\[
\frac{dx}{dt} = \lambda \Phi(x).
\]

\begin{lemma}[Continuity of the vector field]\label{lem:Phi_cont}
Suppose that $F$ is continuous on $\setS$ and that
$\matM(\cdot)$ satisfies Assumption~\ref{ass:metric_regular}.
Then the mapping $\Phi : \setS \to \R^n$ is continuous.
\end{lemma}

\begin{proof}
The idea is to show that each component in the definition of $\Phi(x)$ depends continuously on $x$, and hence their composition is continuous. By Assumption~\ref{ass:metric_regular}, the mapping
$x \mapsto \matM(x)$ is locally Lipschitz on $\setS$, hence continuous,
and so is $x \mapsto \matM^{-1}(x)$ under
Assumption~\ref{ass:metric_bounds}.
Since $F$ is continuous, the mapping
\[
x \mapsto x - \alpha \matM(x)F(x)
\]
is continuous on $\setS$.
For each fixed $x$, the projection
$P_{\setS,\matM^{-1}(x)}(\cdot)$ is continuous, and its dependence
on the metric matrix is continuous under uniform boundedness.
Therefore, the composition defining $\Phi(x)$ is continuous.
\end{proof}

\begin{lemma}[Local existence of solutions]\label{lem:existence}
Under the assumptions of Lemma~\ref{lem:Phi_cont},
for every initial condition $x(0) \in \setS$, there exists
a local solution $x(\cdot)$ of \eqref{eq:SD_dynamics}.
\end{lemma}

\begin{proof}
Since $\Phi(\cdot)$ is continuous on $\setS$,
the result follows from the classical Peano existence theorem
for ordinary differential equations.
\end{proof}

\begin{lemma}[Uniqueness under Lipschitz continuity]\label{lem:unique}
Assume in addition that $F$ is locally Lipschitz continuous on $\setS$.
Then, for every initial condition $x(0) \in \setS$,
the solution of \eqref{eq:SD_dynamics} is unique.
\end{lemma}

\begin{proof}
Under Assumption~\ref{ass:metric_regular}, the mapping
$x \mapsto \matM(x)$ is locally Lipschitz, and hence so are
$x \mapsto \matM^{-1}(x)$ and
$x \mapsto x - \alpha \matM(x)F(x)$.
For fixed $x$, the projection operator
$P_{\setS,\matM^{-1}(x)}(\cdot)$ is Lipschitz continuous in its argument
under Assumption~\ref{ass:metric_bounds}.
Consequently, $\Phi(\cdot)$ is locally Lipschitz on $\setS$,
and uniqueness follows from the Picard--Lindelöf theorem.
\end{proof}

\subsection{Equilibrium points and variational inequalities}

We now show that the equilibria of \eqref{eq:SD_dynamics}
coincide with solutions of the variational inequality
$\mathrm{VI}(F,\setS)$.

\begin{proposition}[Equilibrium--solution equivalence]\label{prop:equilibrium_VI}
A point $\bar x \in \setS$ is an equilibrium of
\eqref{eq:SD_dynamics} if and only if
$\bar x$ solves $\mathrm{VI}(F,\setS)$.
\end{proposition}

\begin{proof}
By definition, $\bar x$ is an equilibrium if and only if
\[
\bar x
=
P_{\setS,\matM^{-1}(\bar x)}
\bigl(\bar x - \alpha \matM(\bar x)F(\bar x)\bigr).
\]
Using the variational characterization
\eqref{eq:proj_char} with
$y = \bar x - \alpha \matM(\bar x)F(\bar x)$,
this condition is equivalent to
\[
\alpha F(\bar x)^T (x - \bar x) \ge 0,
\quad \forall x \in \setS,
\]
which coincides with \eqref{eq:VI}.
\end{proof}

\subsection{Residual mapping}

For later use in the stability analysis, we define the
state-dependent projection residual
\begin{equation}\label{eq:residual}
r(x)
:=
\frac{1}{\alpha}
\bigl(
x -
P_{\setS,\matM^{-1}(x)}
(x - \alpha \matM(x)F(x))
\bigr).
\end{equation}
By Proposition~\ref{prop:equilibrium_VI}, $r(x) = 0$
if and only if $x$ is a solution of $\mathrm{VI}(F,\setS)$.

Having established that equilibria of \eqref{eq:SD_dynamics}  coincide with VI solutions, we now address whether trajectories approach these equilibria. This requires analyzing the stability of the state-dependent projection neural network.

\section{Stability I: Lyapunov stability}\label{sec:stabI}

In this section, we establish Lyapunov stability of equilibria of the
state-dependent projection neural network \eqref{eq:SD_dynamics}.
We deliberately keep the result structural and local, as the precise
constants and sharper convergence claims will be revisited after the
final form of the assumptions on $\matM(\cdot)$ is settled.

\subsection{A nonexpansive fixed-point viewpoint}

Define the state-dependent projected map
\begin{equation}\label{eq:Tmap}
T(x)
:=
P_{\setS,\matM^{-1}(x)}\bigl(x-\alpha \matM(x)F(x)\bigr),
\qquad x \in \setS.
\end{equation}
The map $T$ can be viewed as a single projected forward step with respect to the state-dependent metric: we first move along the direction $-M(x)F(x)$ and then project back onto $S$ using $M(x)^{-1}$. Fixed points of $T$ are exactly the equilibria of the continuous-time system and, by Proposition~\ref{prop:equilibrium_VI}, the solutions of $\mathrm{VI}(F,S)$. Therefore, we can write \eqref{eq:SD_dynamics} as
\begin{equation}\label{eq:SD_dynamics_T}
\frac{dx}{dt} = \lambda\bigl(T(x)-x\bigr).
\end{equation}
By Proposition~\ref{prop:equilibrium_VI}, a point $\bar{x} \in \setS$
is an equilibrium of \eqref{eq:SD_dynamics} if and only if it is a
fixed point of $T$, i.e., $T(\bar{x})=\bar{x}$, equivalently
$\bar{x}$ solves $\mathrm{VI}(F,\setS)$.  The stability proof below relies on a local nonexpansiveness property
of the map $T$ around a fixed point. This isolates the essential
difficulty introduced by the state dependence of the metric:
the map $x \mapsto T(x)$ involves both the projection argument and the
projection metric changing with $x$.

\begin{assumption}[Local nonexpansiveness of $T$]\label{ass:T_nonexp}
Let $\bar{x} \in \setS$ satisfy $T(\bar{x})=\bar{x}$.
There exists a neighborhood $U \subset \setS$ of $\bar{x}$ such that
\begin{equation}\label{eq:T_nonexp}
\|T(x)-T(y)\| \le \|x-y\|,
\qquad \forall x,y \in U.
\end{equation}
\end{assumption}

\begin{remark}\label{rem:T_nonexp_discussion}
Assumption~\ref{ass:T_nonexp} holds automatically in the fixed-metric
case $\matM(\cdot)\equiv M$ when $F$ is monotone and $\alpha$ is chosen
so that the (Euclidean) forward step $x\mapsto x-\alpha M F(x)$ is
nonexpansive; it is also satisfied in when $\alpha$ is chosen conservatively relative to the Lipschitz constants of $F$ and $\matM$ (see Lemma \ref{lem:T_nonexp_sufficient}), and when the metric varies slowly relative to the trajectory velocity. 
% We postpone explicit sufficient conditions (in terms of Lipschitz constants of $F$ and $\matM$) to the refined stability section.
\end{remark}

\subsection{Lyapunov stability of equilibria}
Intuitively, if the state-dependent fixed-point map $T$ does not expand distances near a solution, then trajectories of the continuous-time dynamics should not move away from that solution. The next theorem formalizes this idea using a simple quadratic Lyapunov function.

\begin{theorem}[Lyapunov stability under local nonexpansiveness]\label{thm:Lyap_stability_I}
Assume that \eqref{eq:SD_dynamics} admits an equilibrium point
$\bar{x} \in \setS$ (equivalently, a solution of $\mathrm{VI}(F,\setS)$),
and that Assumption~\ref{ass:T_nonexp} holds in a neighborhood
$U \subset \setS$ of $\bar{x}$. Consider the Lyapunov function
\begin{equation}\label{eq:V_euclid}
V(x) := \frac{1}{2}\|x-\bar{x}\|^2.
\end{equation}
Then, for every trajectory $x(\cdot)$ of \eqref{eq:SD_dynamics_T}
with $x(0)\in U$, the function $t\mapsto V(x(t))$ is nonincreasing as
long as $x(t)\in U$. In particular, $\bar{x}$ is a Lyapunov stable
equilibrium of \eqref{eq:SD_dynamics}.
\end{theorem}

\begin{proof}
Let $\bar{x}$ be a fixed point of $T$, i.e., $T(\bar{x})=\bar{x}$.
Along any differentiable trajectory of \eqref{eq:SD_dynamics_T},
\[
\frac{dV(x(t))}{dt}
= \inner{x(t)-\bar{x}}{\frac{dx}{dt}(t)}
= \lambda\,\inner{x-\bar{x}}{T(x)-x},
\]
where we suppress the time variable for readability.
Rewrite
\[
\inner{x-\bar{x}}{T(x)-x}
=
\inner{x-\bar{x}}{T(x)-T(\bar{x})} - \|x-\bar{x}\|^2.
\]
By Cauchy--Schwarz and Assumption~\ref{ass:T_nonexp},
for all $x\in U$,
\[
\inner{x-\bar{x}}{T(x)-T(\bar{x})}
\le
\|x-\bar{x}\|\,\|T(x)-T(\bar{x})\|
\le
\|x-\bar{x}\|^2.
\]
Hence, for all $x\in U$,
\[
\inner{x-\bar{x}}{T(x)-x} \le 0
\quad \Longrightarrow \quad
\frac{dV(x(t))}{dt}\le 0
\quad \text{whenever } x(t)\in U.
\]
Therefore $V(x(t))$ is nonincreasing as long as the trajectory remains
in $U$. The Lyapunov stability of $\bar{x}$ follows from the standard
Lyapunov stability criterion applied to $V$.
\end{proof}

\begin{remark}\label{rem:Lyap_I_scope}
Theorem~\ref{thm:Lyap_stability_I} is intentionally local and
qualitative. It reduces Lyapunov stability to verifying a
nonexpansiveness property of the state-dependent fixed-point map $T$.
In subsequent sections we will provide explicit sufficient conditions
for \eqref{eq:T_nonexp} in terms of Lipschitz properties of $F$ and
$\matM(\cdot)$, and we will study stronger notions such as asymptotic
and exponential stability under strengthened monotonicity assumptions.
\end{remark}

\begin{remark}
The Lyapunov function $V(x)=\frac12\|x-\bar{x}\|^2$ is chosen for
structural clarity. Since $\matM(\bar{x})$ is symmetric positive
definite, the equivalent quadratic form
$\frac12\|x-\bar{x}\|_{\matM(\bar{x})}^2$ could be used instead without
affecting the stability conclusions.
\end{remark}

\subsection{Sufficient conditions for local nonexpansiveness}

We now provide explicit conditions under which the
state-dependent fixed-point map $T$ defined in \eqref{eq:Tmap}
is locally nonexpansive around a solution of
$\mathrm{VI}(F,\setS)$.

\begin{assumption}[Regularity of $F$]\label{ass:F_Lipschitz}
The mapping $F:\R^n\to\R^n$ is Lipschitz continuous on $\setS$
with constant $K>0$, i.e.,
\[
\|F(x)-F(y)\| \le K\|x-y\|,
\qquad \forall x,y\in\setS.
\]
\end{assumption}

\begin{assumption}[Metric Lipschitz continuity]\label{ass:M_Lipschitz}
The metric mapping $\matM:\R^n\to\R^{n\times n}$ is Lipschitz
continuous on $\setS$ with constant $K_M>0$:
\[
\|\matM(x)-\matM(y)\| \le K_M\|x-y\|,
\qquad \forall x,y\in\setS.
\]
\end{assumption}

\begin{lemma}[Local Lipschitz continuity of the forward step]
\label{lem:forward_Lipschitz}
Under Assumptions~\ref{ass:metric_bounds}, \ref{ass:F_Lipschitz}, and
\ref{ass:M_Lipschitz}, 
the mapping
\[
G(x) := x-\alpha\matM(x)F(x)
\]
is locally Lipschitz on $\setS$, with
\[
\|G(x)-G(y)\|
\le
\bigl(1+\alpha L K + \alpha K_M \|F(\bar{x})\|\bigr)\|x-y\|
+ o(\|x-y\|)
\]
in a neighborhood of any fixed point $\bar{x}$ of $T$.
\end{lemma}

\begin{proof}
For $x,y\in\setS$,
\[
G(x)-G(y)
=
(x-y)
-\alpha\bigl(\matM(x)F(x)-\matM(y)F(y)\bigr).
\]
Add and subtract $\matM(x)F(y)$ to obtain
\[
\matM(x)F(x)-\matM(y)F(y)
=
\matM(x)(F(x)-F(y))
+
(\matM(x)-\matM(y))F(y).
\]
Using Assumptions~\ref{ass:F_Lipschitz},
\ref{ass:M_Lipschitz}, and \ref{ass:metric_bounds},
\[
\|\matM(x)(F(x)-F(y))\|
\le L K \|x-y\|,
\]
and
\[
\|(\matM(x)-\matM(y))F(y)\|
\le K_M \|F(y)\| \|x-y\|.
\]
Since $F$ is continuous and $\bar{x}$ is fixed,
$\|F(y)\| = \|F(\bar{x})\| + o(1)$ as $y\to \bar{x}$.
Combining the above estimates yields the claim.
\end{proof}

\begin{lemma}[Continuity of the metric projection with respect to the metric]
\label{lem:metric_projection_continuity}
Let $S\subset\mathbb{R}^n$ be a nonempty closed convex set, and let
$\mathcal Q\subset\mathbb{S}_{++}^n$ be a compact set of symmetric positive
definite matrices. Then the mapping
\[
(Q,z)\mapsto P_{S,Q}(z)
\]
is locally Lipschitz continuous on $\mathcal Q\times\mathbb{R}^n$.
In particular, for every compact set $K\subset\mathbb{R}^n$, there exists
a constant $C>0$ such that
\[
\|P_{S,Q_1}(z)-P_{S,Q_2}(z)\|
\le C\|Q_1-Q_2\|,
\quad
\forall z\in K,\ \forall Q_1,Q_2\in\mathcal Q.
\]
\end{lemma}

\begin{proof}
Fix $z\in\mathbb{R}^n$ and $Q\in\mathcal Q$. The metric projection
$P_{S,Q}(z)$ is the unique solution of the strongly convex optimization
problem
\[
\min_{y\in S}\ \frac12 (y-z)^\top Q (y-z).
\]
Since $\mathcal Q$ is compact, there exist constants $0<m\le M<\infty$
such that
\[
m I \preceq Q \preceq M I
\quad \forall Q\in\mathcal Q.
\]
(The notation $A\preceq B$ means that $B-A$ is symmetric positive semidefinite.) Hence, the objective function is uniformly strongly convex with modulus $m$
and has Lipschitz continuous gradient with Lipschitz constant $M$. Let $Q_1,Q_2\in\mathcal Q$, and denote $y_i:=P_{S,Q_i}(z)$ for $i=1,2$.
By the first-order optimality conditions, there exist
$\xi_i\in N_S(y_i)$ such that
\[
Q_i(y_i-z)+\xi_i=0,\quad i=1,2,
\]
where $N_S(\cdot)$ denotes the normal cone to $S$.
Subtracting the two conditions and taking the inner product with
$y_1-y_2$ yields
\[
\langle Q_1(y_1-y_2),y_1-y_2\rangle
=
\langle (Q_2-Q_1)(y_2-z),y_1-y_2\rangle
-\langle \xi_1-\xi_2,y_1-y_2\rangle.
\]
By monotonicity of the normal cone mapping, the last term is nonnegative.
Using uniform positive definiteness of $Q_1$ and Cauchy--Schwarz inequality,
we obtain
\[
m\|y_1-y_2\|^2
\le
\|Q_1-Q_2\|\,\|y_2-z\|\,\|y_1-y_2\|.
\]
Since $z$ is restricted to a compact set and projections onto $S$ are
bounded on compact sets, $\|y_2-z\|$ is uniformly bounded. Dividing both
sides by $\|y_1-y_2\|$ (when nonzero) yields
\[
\|y_1-y_2\|\le C\|Q_1-Q_2\|
\]
for some constant $C>0$ independent of $Q_1,Q_2$. This proves the claimed
local Lipschitz continuity with respect to $Q$.
\end{proof}

\begin{lemma}[Local nonexpansiveness of $T$]
\label{lem:T_nonexp_sufficient}
Let $\bar{x}\in\setS$ be a solution of $\mathrm{VI}(F,\setS)$.
Suppose Assumptions~\ref{ass:metric_bounds},
\ref{ass:F_Lipschitz}, and \ref{ass:M_Lipschitz} hold.
If $\alpha>0$ satisfies
\begin{equation}\label{eq:stepsize_condition}
\alpha\bigl(LK + K_M\|F(\bar{x})\|\bigr) < 1,
\end{equation}
then there exists a neighborhood $U$ of $\bar{x}$ such that
\[
\|T(x)-T(y)\| \le \|x-y\|,
\qquad \forall x,y\in U.
\]
\end{lemma}

\begin{proof}
Fix $\bar{x}\in\setS$ with $T(\bar{x})=\bar{x}$.
For $x,y$ sufficiently close to $\bar{x}$, write
\[
T(x)-T(y)
=
P_{\setS,\matM^{-1}(x)}(G(x))
-
P_{\setS,\matM^{-1}(y)}(G(y)).
\]
Add and subtract $P_{\setS,\matM^{-1}(x)}(G(y))$:
\[
\|T(x)-T(y)\|
\le
\|P_{\setS,\matM^{-1}(x)}(G(x))
-
P_{\setS,\matM^{-1}(x)}(G(y))\|
+
\Delta_M,
\]
where
\[
\Delta_M :=
\|P_{\setS,\matM^{-1}(x)}(G(y))
-
P_{\setS,\matM^{-1}(y)}(G(y))\|.
\]

The first term is bounded using Lipschitz continuity of the
projection with fixed metric (uniformly under
Assumption~\ref{ass:metric_bounds}):
\[
\|P_{\setS,\matM^{-1}(x)}(G(x))
-
P_{\setS,\matM^{-1}(x)}(G(y))\|
\le L^2\|G(x)-G(y)\|.
\]
By Lemma~\ref{lem:forward_Lipschitz},
this term is bounded by
\[
L^2\bigl(1+\alpha L K + \alpha K_M\|F(\bar{x})\|\bigr)\|x-y\|
+ o(\|x-y\|).
\]

The second term $\Delta_M$ arises from variation of the metric.
Under Assumption~\ref{ass:M_Lipschitz}, continuity of the projection
with respect to the metric implies
\[
\Delta_M = o(\|x-y\|)
\quad \text{as } x,y\to \bar{x}.
\]

Combining the estimates and using the step-size condition
\eqref{eq:stepsize_condition}, we conclude that
\[
\|T(x)-T(y)\| \le \|x-y\|
\]
for all $x,y$ in a sufficiently small neighborhood of $\bar{x}$.
\end{proof}

\begin{remark}\label{rem:stepsize_interpretation}
Condition \eqref{eq:stepsize_condition} has a natural interpretation:
the effective preconditioning strength $\alpha\matM(x)$ must be small
relative to both the Lipschitz variation of $F$ and the local
variation of the metric. In particular, when $\matM(\cdot)$ is constant,
the condition reduces to the classical small-step requirement
$\alpha L K < 1$.
\end{remark}

\section{Stability II: Asymptotic stability}\label{sec:stabII}

We now strengthen the Lyapunov stability result of the previous section
by establishing asymptotic convergence of trajectories to the solution
set of $\mathrm{VI}(F,\setS)$ under a monotonicity assumption on $F$.

\subsection{Monotonicity assumption}

\begin{assumption}[Monotonicity of $F$]\label{ass:monotone}
The mapping $F:\R^n\to\R^n$ is monotone on $\setS$, i.e.,
\[
\inner{F(x)-F(y)}{x-y} \ge 0,
\qquad \forall x,y \in \setS.
\]
\end{assumption}

This assumption is standard in the theory of variational inequalities
and ensures convexity of the solution set of $\mathrm{VI}(F,\setS)$.

\subsection{Decrease of the Lyapunov function}

Let $\bar{x} \in \setS$ be a solution of $\mathrm{VI}(F,\setS)$, and
consider the Lyapunov function
\[
V(x) := \frac{1}{2}\|x-\bar{x}\|^2.
\]
As shown in Theorem~\ref{thm:Lyap_stability_I}, under local
nonexpansiveness of $T$, $V$ is nonincreasing along trajectories of
\eqref{eq:SD_dynamics_T}.

The next lemma identifies the set where $\frac{dV}{dt}$ vanishes.

\begin{lemma}[Characterization of $\frac{dV}{dt} = 0$]\label{lem:Vdot_zero}
Suppose Assumption~\ref{ass:monotone} holds.
Along any trajectory $x(\cdot)$ of \eqref{eq:SD_dynamics_T},
\[
\frac{dV(x(t))}{dt} = 0
\quad \Longleftrightarrow \quad
T(x(t)) = x(t).
\]
\end{lemma}

\begin{proof}
We use the fact that $V$ can only stop decreasing along a trajectory if the map $T$ stops pulling $x$ toward the solution set, which by our equilibrium characterization means that $x$ is already a fixed point of $T$. From the proof of Theorem~\ref{thm:Lyap_stability_I},
\[
\frac{dV(x(t))}{dt}
=
\lambda\,\inner{x-\bar{x}}{T(x)-x}.
\]
We have already shown that
\[
\inner{x-\bar{x}}{T(x)-x} \le 0,
\]
with equality if and only if
\[
\inner{x-\bar{x}}{T(x)-T(\bar{x})} = \|x-\bar{x}\|^2.
\]
By Cauchy--Schwarz, equality occurs if and only if
\[
T(x)-T(\bar{x}) = x-\bar{x},
\]
and since $T(\bar{x})=\bar{x}$, this is equivalent to $T(x)=x$.

Finally, by Proposition~\ref{prop:equilibrium_VI}, $T(x)=x$ holds if and
only if $x$ solves $\mathrm{VI}(F,\setS)$.
\end{proof}

\subsection{Asymptotic convergence}

Lyapunov stability guarantees that trajectories do not diverge from a solution, but it does not by itself ensure convergence. Under a monotonicity assumption on $F$, we can strengthen the result by combining the Lyapunov function with LaSalle’s invariance principle to obtain asymptotic convergence.

\begin{theorem}[Asymptotic stability]\label{thm:asymptotic}
Suppose Assumptions~\ref{ass:metric_bounds},
\ref{ass:F_Lipschitz}, \ref{ass:M_Lipschitz},
and \ref{ass:monotone} hold, and let $\bar{x}$ be a solution of
$\mathrm{VI}(F,\setS)$.
Assume that the step-size condition \eqref{eq:stepsize_condition} holds
so that the map $T$ is locally nonexpansive around $\bar{x}$.
Then $\bar{x}$ is an asymptotically stable equilibrium of
\eqref{eq:SD_dynamics}.

Moreover, every trajectory $x(\cdot)$ of \eqref{eq:SD_dynamics}
starting sufficiently close to $\bar{x}$ converges to the solution set
of $\mathrm{VI}(F,\setS)$.
\end{theorem}

\begin{proof}
By Theorem~\ref{thm:Lyap_stability_I}, $\bar{x}$ is Lyapunov stable and
the function $V(x(t))$ is nonincreasing along trajectories starting in
a neighborhood $U$ of $\bar{x}$.
Hence trajectories remain in a compact subset of $U$.

By Lemma~\ref{lem:Vdot_zero}, the largest invariant set contained in
\[
\{x \in U : \frac{dV(x(t))}{dt} = 0\}
\]
is precisely the set of fixed points of $T$, which coincides with the
solution set of $\mathrm{VI}(F,\setS)$.
LaSalle’s invariance principle therefore implies that
\[
\lim_{t\to\infty} \operatorname{dist}\bigl(x(t), \operatorname{SOL}(F,\setS)\bigr) = 0,
\]
where $\operatorname{SOL}(F,\setS)$ denotes the solution set of the
variational inequality.
\end{proof}

\begin{remark}\label{rem:asymptotic_scope}
Theorem~\ref{thm:asymptotic} establishes asymptotic convergence without
requiring strong monotonicity of $F$ or a fixed metric.
The result highlights the role of state-dependent metric regularity in
preserving the invariance structure necessary for LaSalle-type
arguments. Stronger assumptions will be required to obtain convergence
rates, which is the subject of the next section.
\end{remark}

\section{Stability III: Exponential convergence}\label{sec:stabIII}

In this section we strengthen the asymptotic convergence result of
Section~4 by establishing local exponential convergence of trajectories
of the state-dependent projection neural network under a strong monotonicity
assumption on $F$.

\subsection{Strong monotonicity}

\begin{assumption}[Strong monotonicity]\label{ass:strong_monotone}
The mapping $F:\R^n\to\R^n$ is strongly monotone on $\setS$ with modulus
$\mu>0$, i.e.,
\[
\inner{F(x)-F(y)}{x-y} \ge \mu\|x-y\|^2,
\qquad \forall x,y\in\setS.
\]
\end{assumption}

Under Assumption~\ref{ass:strong_monotone}, the variational inequality
$\mathrm{VI}(F,\setS)$ admits a unique solution, which we denote by
$\bar{x}$.

\subsection{A contractive estimate}

Recall the fixed-point representation
\[
\frac{dx}{dt} = \lambda\bigl(T(x)-x\bigr),
\qquad
T(x) =
P_{\setS,\matM^{-1}(x)}\bigl(x-\alpha \matM(x)F(x)\bigr).
\]

The following lemma refines the nonexpansiveness result obtained in
Section~4 into a strict contraction estimate near $\bar{x}$.

\begin{lemma}[Local contraction of $T$]\label{lem:T_contraction}
Suppose Assumptions~\ref{ass:metric_bounds},
\ref{ass:F_Lipschitz}, \ref{ass:M_Lipschitz},
and \ref{ass:strong_monotone} hold.
Then there exist constants $\rho \in (0,1)$ and a neighborhood
$U \subset \setS$ of $\bar{x}$ such that
\[
\|T(x)-\bar{x}\| \le \rho \|x-\bar{x}\|,
\qquad \forall x\in U.
\]
\end{lemma}

\begin{proof}
For $x$ sufficiently close to $\bar{x}$, write
\[
\|T(x)-\bar{x}\|
=
\|P_{\setS,\matM^{-1}(x)}(x-\alpha\matM(x)F(x))
-
P_{\setS,\matM^{-1}(\bar{x})}(\bar{x}-\alpha\matM(\bar{x})F(\bar{x}))\|.
\]
Using Lipschitz continuity of the projection with respect to its
argument (uniformly under Assumption~\ref{ass:metric_bounds}) and
continuity with respect to the metric, we obtain
\[
\|T(x)-\bar{x}\|
\le
L^2\|x-\bar{x}-\alpha(\matM(x)F(x)-\matM(\bar{x})F(\bar{x}))\|
+ o(\|x-\bar{x}\|).
\]

As in Lemma~\ref{lem:forward_Lipschitz}, decompose
\[
\matM(x)F(x)-\matM(\bar{x})F(\bar{x})
=
\matM(x)(F(x)-F(\bar{x}))
+
(\matM(x)-\matM(\bar{x}))F(\bar{x}).
\]
By strong monotonicity,
\[
\inner{F(x)-F(\bar{x})}{x-\bar{x}}
\ge
\mu\|x-\bar{x}\|^2,
\]
and by Lipschitz continuity of $F$ and $\matM(\cdot)$,
\[
\|\matM(x)(F(x)-F(\bar{x}))\| \le L K \|x-\bar{x}\|,
\qquad
\|(\matM(x)-\matM(\bar{x}))F(\bar{x})\|
\le K_M\|F(\bar{x})\|\|x-\bar{x}\|.
\]
Combining the above estimates yields
\[
\|T(x)-\bar{x}\|
\le
\bigl(1-\alpha\mu + \alpha c\bigr)\|x-\bar{x}\|
\]
for some constant $c>0$ depending on $L,K,K_M,\|F(\bar{x})\|$.
Choosing $\alpha>0$ sufficiently small ensures $\rho:=1-\alpha\mu+\alpha c<1$,
which completes the proof.
\end{proof}

\subsection{Exponential stability}

We now establish the main result of this section. When $F$ is strongly monotone, the variational inequality has a unique solution and the dynamics should contract toward it at a definite rate. The next theorem shows that, under the same regularity and small-step conditions, the state-dependent projection neural network are locally exponentially convergent.

\begin{theorem}[Local exponential stability]\label{thm:exponential}
Suppose Assumptions~\ref{ass:metric_bounds},
\ref{ass:F_Lipschitz}, \ref{ass:M_Lipschitz},
and \ref{ass:strong_monotone} hold.
Then there exist constants $\gamma>0$ and $\delta>0$ such that every
trajectory $x(\cdot)$ of \eqref{eq:SD_dynamics} with
$\|x(0)-\bar{x}\|<\delta$ satisfies
\[
\|x(t)-\bar{x}\| \le \|x(0)-\bar{x}\| e^{-\gamma t},
\qquad \forall t\ge 0.
\]
\end{theorem}

\begin{proof}
Let $V(x)=\frac{1}{2}\|x-\bar{x}\|^2$.
Along trajectories of \eqref{eq:SD_dynamics_T},
\[
\frac{dV(x(t))}{dt}
=
\lambda\,\inner{x-\bar{x}}{T(x)-x}.
\]
By Lemma~\ref{lem:T_contraction}, for $x$ in a neighborhood of $\bar{x}$,
\[
\inner{x-\bar{x}}{T(x)-x}
=
\inner{x-\bar{x}}{T(x)-\bar{x}} - \|x-\bar{x}\|^2
\le
-(1-\rho)\|x-\bar{x}\|^2.
\]
Thus
\[
\frac{dV(x(t))}{dt}\le -2\lambda(1-\rho)V(x).
\]
Setting $\gamma := \lambda(1-\rho)$ and applying Gronwall’s inequality
yields the claimed exponential decay.
\end{proof}

\begin{remark}\label{rem:exp_scope}
The exponential convergence result is local and relies on a small-step
regime ensuring contraction of the state-dependent fixed-point map $T$.
The rate $\gamma$ depends explicitly on the strong monotonicity modulus
of $F$ and on the regularity of the metric $\matM(\cdot)$.
Global exponential convergence may fail in general when the metric
varies with the state.
\end{remark}

\begin{remark}
The constants $\rho$, $\gamma$, and $\delta$ appearing in
Lemma~\ref{lem:T_contraction} and Theorem~\ref{thm:exponential}
are established in an existential manner and are not intended to be
computed explicitly. Their role is to certify local contraction and
exponential stability under sufficiently small step sizes and metric
variations. Obtaining explicit expressions for these constants would
require global bounds on the Lipschitz constants of $F$ and $M(\cdot)$,
as well as detailed spectral information on the state-dependent metric,
which is beyond the scope of the present analysis.
\end{remark}

\section{Illustrative example}\label{sec:example}

We present a simple variational inequality for which the solution is
unique and can be characterized explicitly, and we illustrate how a
state-dependent metric can be interpreted as a continuous-time
preconditioner.

Let $\setS := [0,1]^n$ and consider the affine mapping
\[
F(x) := Qx + q,
\]
where $Q \in \R^{n\times n}$ is symmetric positive definite and
$q\in\R^n$. Then $F$ is strongly monotone on $\R^n$ with modulus
$\lambda_{\min}(Q)$, hence $\mathrm{VI}(F,\setS)$ has a unique solution
$\bar{x} \in \setS$.

Moreover, $\bar{x}$ is the unique minimizer of the strongly convex
quadratic program
\[
\min_{x\in\setS} \ \frac12 x^T Q x + q^T x,
\]
and therefore it satisfies the first-order optimality condition
\[
\inner{Q\bar{x} + q}{x-\bar{x}} \ge 0,
\quad \forall x\in\setS,
\]
which coincides with $\mathrm{VI}(F,\setS)$.

We consider the state-dependent projection neural network
\[
\frac{dx}{dt}
=
\lambda\Bigl[
P_{\setS,\matM^{-1}(x)}\bigl(x-\alpha\matM(x)(Qx+q)\bigr)-x
\Bigr].
\]
As a representative example of a smoothly varying diagonal metric, we
set
\[
\matM(x) := \diag(m_1(x),\dots,m_n(x)),
\qquad
m_i(x) := m_{\min} + (m_{\max}-m_{\min})\bigl(4x_i(1-x_i)\bigr),
\]
where $0<m_{\min}\le m_{\max}$.
This choice satisfies Assumption~\ref{ass:metric_bounds} on $\setS$ and
provides a simple mechanism for coordinate-wise scaling that varies
with the state.

\section{Numerical Experiments}\label{sec:numerical}

The goal of this numerical section is to illustrate the geometric mechanisms induced by state-dependent metrics, rather than to benchmark performance against other algorithms. We use low-dimensional examples to visualize how different metric choices shape trajectories, interact with constraints, and influence transient behavior, while all methods converge to the same solution of the underlying variational inequality.

All numerical examples are conducted in regimes where the assumptions of
the theoretical analysis are satisfied locally along the generated
trajectories. In particular, the state-dependent metrics are smooth and
uniformly positive definite on the compact regions explored by the
solutions, and the step-size parameters are chosen conservatively to
ensure contraction of the induced fixed-point map, in accordance with
Lemma~\ref{lem:T_contraction} and Theorem~\ref{thm:exponential}. The focus
is therefore on qualitative geometric behavior rather than global
performance guarantees. 

The numerical experiments are implemented in \textsc{Julia}~1.12.3 and
executed on a personal computer equipped with an Intel
Core\texttrademark~i9-9900K processor (3.6~GHz) and 32~GB of RAM. All
computations are performed in double precision. The reported results are
fully reproducible and are intended to complement, rather than extend,
the scope of the theoretical analysis.
\subsection{Example 1: Geometry-aware behavior under state-dependent metrics}

We consider a two-dimensional strongly monotone variational inequality designed to
highlight geometric effects induced by state-dependent metrics.
The example is kept low-dimensional to allow direct visualization and to isolate
metric effects from other algorithmic factors.
In \eqref{eq:VI}, the feasible set is
\[
\setS := [-1,1]^2 \subset \mathbb{R}^2,
\]
and the operator is affine, $F(x)=Ax-b$, where
\[
A = R(\phi)^\top
\begin{pmatrix}
\kappa & 0 \\
0 & 1
\end{pmatrix}
R(\phi),
\qquad
R(\phi)=
\begin{pmatrix}
\cos\phi & -\sin\phi \\
\sin\phi & \cos\phi
\end{pmatrix}.
\]
We fix $\kappa=50$ and $\phi=\pi/6$, so that $A$ is symmetric positive definite,
strongly anisotropic, and rotated with respect to the coordinate axes.
Consequently, $F$ is strongly monotone and Lipschitz continuous, and
$\mathrm{VI}(F,\setS)$ admits a unique solution.

To make the solution explicit, we prescribe
\[
\bar{x} := (0.8,-0.6)\in \operatorname{int}(\setS),
\qquad
b := A \bar{x},
\]
so that $\bar{x}$ is the unique solution of $\mathrm{VI}(F,\setS)$.

\paragraph{Metrics under comparison.}
We compare the projected dynamics under three metric choices:
\begin{enumerate}
\item the Euclidean metric $\matM(x)\equiv I$;
\item a fixed metric $\matM(\bar{x})$, obtained by freezing the state-dependent metric at the solution;
\item a state-dependent metric
\[
\matM(x)
=
R(\theta(x))^\top
\begin{pmatrix}
1 + r^2 & 0 \\
0 & 2 + r^2
\end{pmatrix}
R(\theta(x)),
\qquad
r = \|x\|,
\]
where $\theta(x)=\arctan2(x_2,x_1)$.
\end{enumerate}
The state-dependent metric introduces position-dependent anisotropy and rotation:
its principal directions rotate along trajectories and its conditioning increases
with distance from the origin, behavior that cannot be reproduced by any fixed
positive definite matrix.

For each metric, we simulate the SD-SPNN dynamics \eqref{eq:SD_dynamics}
from identical initial points and with identical parameters.
Unless stated otherwise, the step size is fixed to $\alpha=0.1$, which lies in the
conservative regime suggested by the local contraction analysis and yields stable
behavior across all metrics.
Larger step sizes are considered separately as stress tests beyond the locally
guaranteed regime.

The purpose of this experiment is not performance comparison or step-size tuning,
but to examine how metric choice shapes the geometry of the induced flow.
We focus on trajectory curvature, orientation, boundary interaction, and the way
state-dependent metrics adapt the geometry along trajectories.
All three dynamics share the same equilibrium and local convergence guarantees;
the differences observed below are therefore transient and geometric in nature.

\subsection{Trajectory geometry and metric-induced flow adaptation}

We first examine how metric choice affects the geometric structure of SD-SPNN
trajectories.
All experiments use $\alpha=0.1$ and identical initial conditions located at
corners and edge midpoints of $\setS=[-1,1]^2$, a regime yielding stable behavior
for all metrics.

Figure~\ref{fig:traj_compare} shows the resulting trajectories.
Although all trajectories converge to the same equilibrium $\bar{x}$, their
transient behavior and boundary interactions differ substantially.

Under the Euclidean metric, trajectories exhibit pronounced piecewise-linear motion
with frequent direction changes at the boundary.
Corner initializations tend to slide along faces and alternate between
coordinate-aligned directions, producing characteristic zig--zagging behavior.
This reflects the mismatch between Euclidean geometry and the strong anisotropy
of $F$.

With the fixed metric $\matM(\bar{x})$, trajectories are noticeably smoother.
The fixed metric partially compensates for anisotropy and reduces boundary
oscillations, yielding more direct paths toward $\bar{x}$.
However, because the metric orientation and scaling are constant, the geometry
cannot adapt to the local state, and early transients may still deviate from
coherent descent directions.

In contrast, the state-dependent metric adapts continuously to the evolving state.
Its principal directions rotate and rescale along trajectories, producing smoothly
curved paths that steer away from repeated boundary reflections and toward $\bar{x}$.
This adaptive behavior is most evident for corner initializations and cannot be
replicated by any fixed metric.

\begin{figure}[t]
    \centering
    \includegraphics[width=\textwidth]{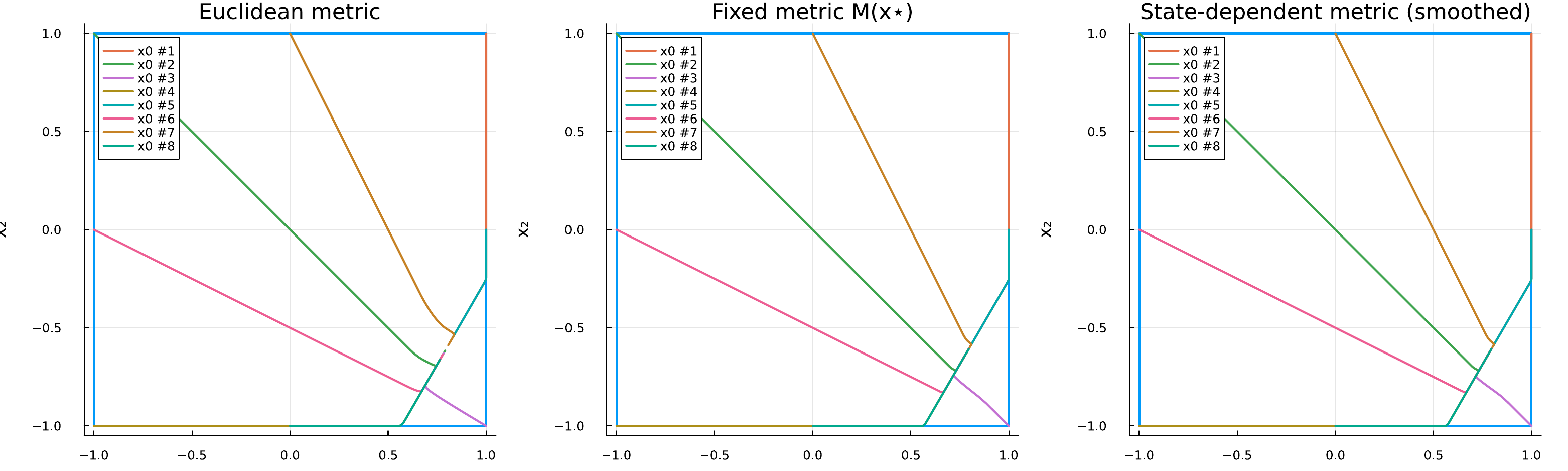}
    \caption{Trajectories of the SD-SPNN dynamics for $\alpha=0.1$ under three
    metric choices, starting from identical corner and edge initial points.
    Left: Euclidean metric. Center: fixed metric $\matM(\bar{x})$.
    Right: state-dependent metric (smoothed). All trajectories converge to
    $\bar{x}$ but exhibit distinct transient geometric behavior.}
    \label{fig:traj_compare}
\end{figure}

Table~\ref{tab:geom_stats} summarizes quantitative indicators over $N=50$ random
initial points uniformly distributed on $\setS$.
For this conservative step size, no trajectory reaches the strict tolerance
$\|r(x)\|\le10^{-6}$ within the time horizon; nevertheless, all metrics achieve
substantial residual reduction and Lyapunov decrease.
The fixed and state-dependent metrics attain significantly smaller final residuals
and Lyapunov values than the Euclidean metric, at the cost of increased metric
conditioning, reflecting their anisotropic and adaptive nature.

\begin{table}[t]
\centering
\caption{Summary statistics for random initializations ($N=50$) with $\alpha=0.1$.
Success is defined by $r_{\text{final}}\le10^{-2}$. Reported values are medians
across runs, except for the maximum condition number $\kappa_{M}$, reported as
the maximum observed along trajectories (then maximized across runs).}
\label{tab:geom_stats}
\begin{tabular}{lccccc}
\toprule
Metric & Success rate & Median $r_{\mathrm{final}}$ & Median $V_{\mathrm{final}}$
& Max $\kappa_M$ & Mean active constraints \\
\midrule
euclid & 13.8\% & $3.68e-02$ & $6.76e-04$ & 1.00 & 0.006 \\
fixed & 51.7\% & $9.33e-03$ & $9.79e-06$ & 1.50 & 0.004 \\
sd\_raw & 51.7\% & $9.22e-03$ & $9.59e-06$ & 2.00 & 0.004 \\
sd\_smooth & 51.7\% & $9.22e-03$ & $9.59e-06$ & 2.00 & 0.004 \\
\bottomrule
\end{tabular}
\end{table}

The differences in Figure~\ref{fig:traj_compare} and Table~\ref{tab:geom_stats}
should not be interpreted as differences in asymptotic convergence, since all
metrics share the same equilibrium and local guarantees.
Rather, they illustrate how metric choice shapes the transient geometry of the
projected flow, with state-dependent metrics providing directional adaptation
unattainable by Euclidean or fixed metrics.

\paragraph{Rotated-box robustness check.}
To verify that the geometric effects in Figure~\ref{fig:traj_compare} are not an artifact of
axis-aligned clipping, we repeat the trajectory experiment on the rotated box
\(\mathcal S_\theta := R(\theta)[-1,1]^2\) with \(\theta=\pi/8\).
All other ingredients are unchanged: the operator, metrics, step size \(\alpha=0.1\),
and deterministic initializations (mapped onto \(\mathcal S_\theta\)).
Figure~\ref{fig:traj_compare_rotbox} exhibits the same qualitative separation between metrics:
the Euclidean flow shows more pronounced boundary-induced direction changes,
the fixed metric partially mitigates these effects, and the state-dependent metric
produces the smoothest geometry.
This confirms that the observed behavior reflects metric-induced shaping of the projected flow,
rather than a coordinate-aligned projection artifact.
\begin{figure}[t]
    \centering
    \includegraphics[width=\textwidth]{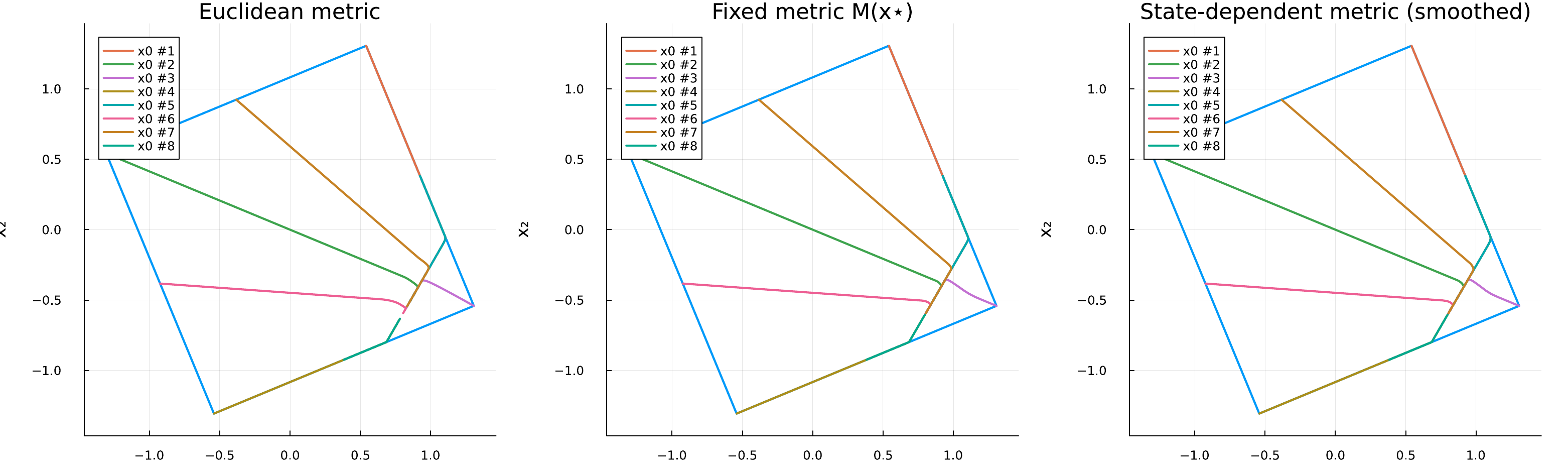}
    \caption{Trajectories for the same setup as Figure~\ref{fig:traj_compare}, but on the rotated
    box \(\mathcal S_\theta = R(\theta)[-1,1]^2\) with \(\theta=\pi/8\).
    Left: Euclidean metric. Center: fixed metric \(\matM(\bar{x})\).
    Right: state-dependent metric. The qualitative geometric conclusions persist
    under non-axis-aligned constraints.}
    \label{fig:traj_compare_rotbox}
\end{figure}

\subsection{Example 2 : Region-based eigenvector rotation}
\label{subsec:ex2_region_rotation}

This flagship example is designed to isolate a geometric effect that cannot be reproduced by any
fixed SPD metric: \emph{state-dependent rotation of the metric eigenvectors}.
On the box constraint $\mathcal S=[-1,1]^2$, we equip the SD-SPNN mapping with a smoothly varying
angle field
\begin{equation}\label{eq:ex2_theta_region}
\theta(x)=\phi+\omega\,\tanh(\beta x_1),
\end{equation}
and define the associated state-dependent metric
\begin{equation}\label{eq:ex2_M_region}
M_{\mathrm{reg}}(x)=R(\theta(x))^\top \Lambda\,R(\theta(x)),
\qquad
\Lambda=\mathrm{diag}(1,\gamma),
\end{equation}
where $R(\theta)$ is the planar rotation matrix. In contrast to any fixed SPD matrix (whose
eigendirections are constant in space), the eigenvectors of $M_{\mathrm{reg}}(x)$ vary with $x_1$,
thereby changing the local geometry experienced by the projection step.
Throughout, we set $\phi=0$, $\omega=3\pi/4$, $\beta=3$, and $\gamma=8$, so that the anisotropy is
uniformly bounded (constant $\kappa_M=\gamma$) while the eigenframe rotates with the state.

We run the continuous-time SD-SPNN dynamics with a conservative step size $\alpha=0.1$ on the finite
horizon $t\in[0,T]$ (same integrator and tolerances as in Example~1), and we compare three geometries:
(i) Euclidean ($M\equiv I$), (ii) a fixed SPD metric obtained by freezing \eqref{eq:ex2_M_region} at the
solution $\bar{x}$, and (iii) the state-dependent region metric $M_{\mathrm{reg}}(x)$.
All trajectories are projected onto $\mathcal S$ using the same metric projection routine.

Figure~\ref{fig:ex2_traj_region} reports a trajectory comparison from deterministic initializations.
The Euclidean and fixed-metric baselines impose a globally uniform geometry: the former is isotropic,
while the latter introduces a constant anisotropic bias through a fixed eigenframe. In contrast, the
state-dependent region metric produces smoothly varying direction changes along the trajectories,
reflecting the fact that the preferred directions (metric eigenvectors) rotate with the state. Since
$\Lambda$ is constant in \eqref{eq:ex2_M_region}, these differences are attributable to eigenvector
rotation rather than state-dependent scaling. All trajectories converge to the same equilibrium; the observed differences reflect transient geometric effects induced by the metric choice, consistent with local stability theory.

To directly verify the mechanism, Figure~\ref{fig:ex2_theta_region} plots the evolution of the angle
$\theta(x(t))$ along the state-dependent trajectories. The observed variation in $\theta(x(t))$ across
time confirms that the metric eigenframe changes along the flow, providing a direct explanation for
the geometric separation in Figure~\ref{fig:ex2_traj_region}. Consistent with the Numerics Blueprint,
we make no global superiority claims; the point of this example is to exhibit a geometry that a fixed
SPD metric cannot mimic.

All metrics converge to the same equilibrium; the observed differences reflect transient geometric effects induced by the metric choice, consistent with local stability theory.

\begin{figure}[t]
  \centering
  \includegraphics[width=0.82\linewidth]{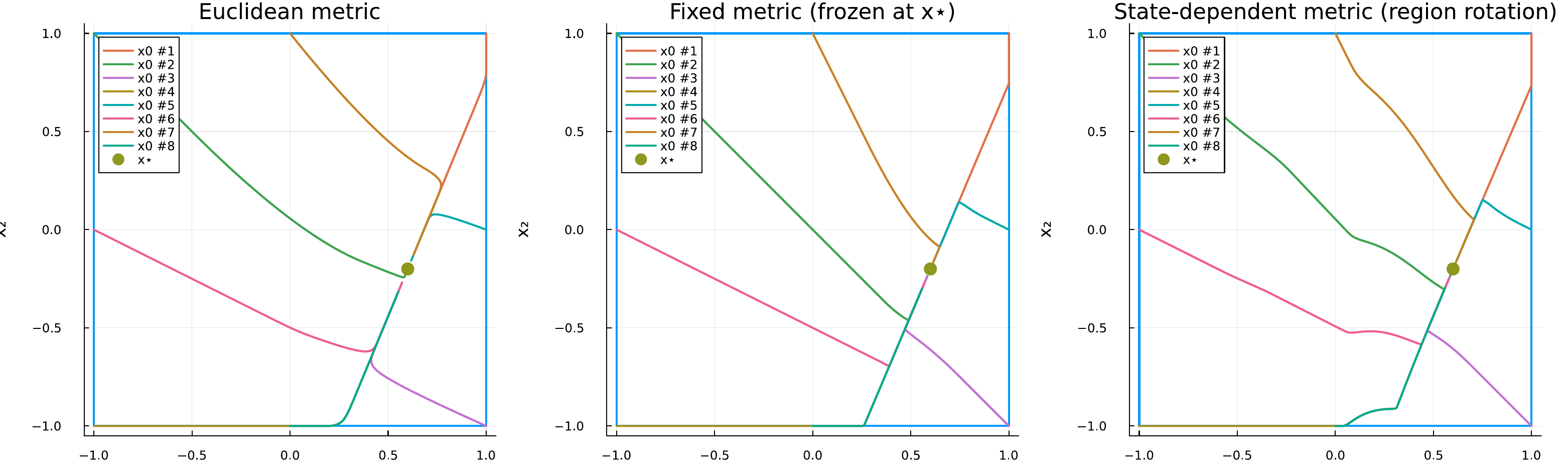}
  \caption{Example 2 (flagship): trajectory geometry on $\mathcal S=[-1,1]^2$ for Euclidean, fixed SPD
  (region metric frozen at $\bar{x}$), and the state-dependent region metric \eqref{eq:ex2_M_region}
  with angle field \eqref{eq:ex2_theta_region}. The state-dependent metric yields smoother trajectories and altered transient geometry, while converging to the same solution.}
  \label{fig:ex2_traj_region}
\end{figure}

\begin{figure}[t]
  \centering
  \includegraphics[width=0.72\linewidth]{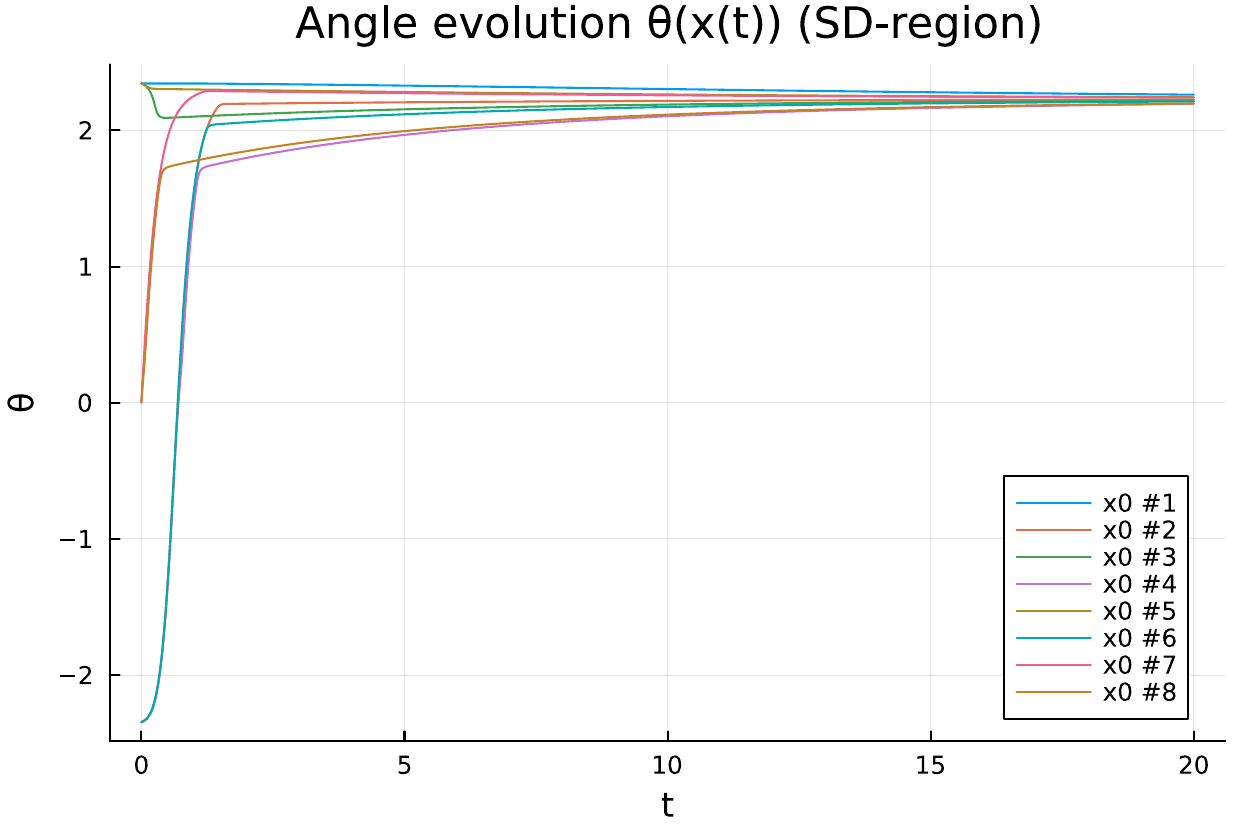}
  \caption{Example 2 diagnostic: Evolution of the metric eigenframe angle $\theta(x(t))$, illustrating state-dependent rotation rather than convergence acceleration. The time variation of $\theta(x(t))$ confirms
  that the eigenframe rotates along the flow, explaining the geometric effects observed in
  Figure~\ref{fig:ex2_traj_region}.}
  \label{fig:ex2_theta_region}
\end{figure}

\subsection{Example 3: Nonlinear monotone VI (interior solution)}

We consider a nonlinear variational inequality on the box
$\mathcal S=[-1,1]^2$ with operator
\[
F(x)=Qx-b+\gamma\bigl(e^{-c x}-\mathbf{1}\bigr),
\]
where $Q=R(\varphi_Q)^\top\mathrm{diag}(\kappa,1)R(\varphi_Q)$ is symmetric
positive definite.
The parameters are $\kappa=25$, $\varphi_Q=0.314$, $\gamma=0.3$, and $c=2.0$.
The right-hand side $b$ is chosen so that the unique solution is
$\bar{x}=(0.4,-0.3)$, which lies strictly in the interior of $\mathcal S$.

We compare three geometries: the Euclidean metric, a fixed metric frozen at
$M(\bar{x})$, and a state-dependent metric whose eigenvectors rotate with the
state. All experiments use the same step size $\alpha=0.1$ and are integrated
until the residual norm satisfies $\|r(x)\|\le10^{-6}$.

\begin{figure}
    \centering
    \includegraphics[width=0.75\linewidth]{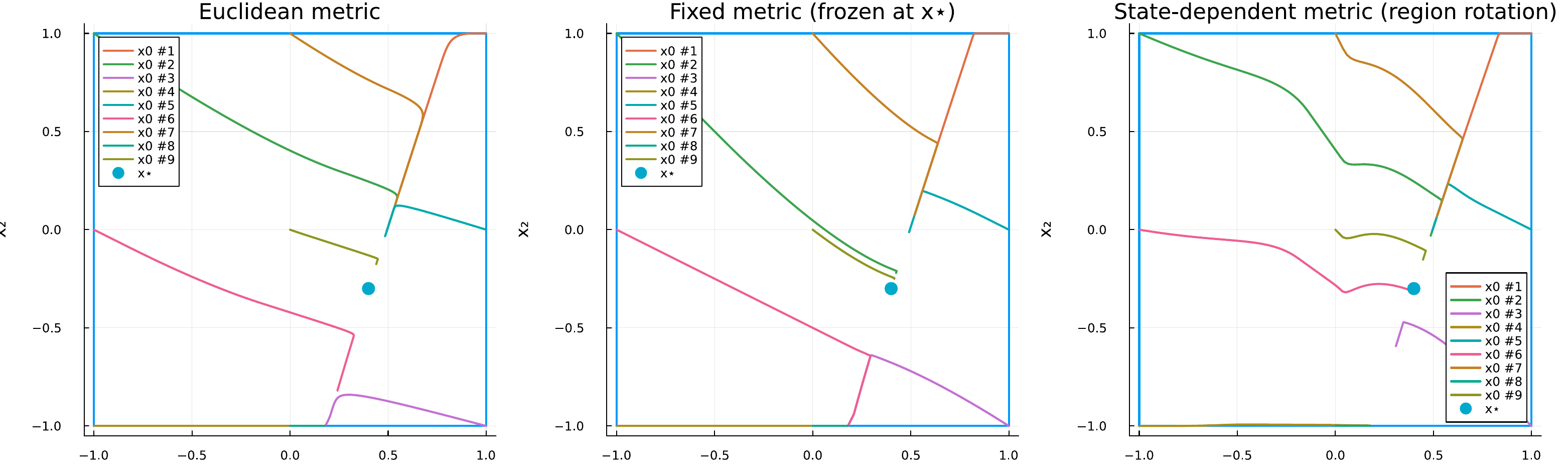}
    \caption{Trajectory comparison for the nonlinear interior solution.
Left: Euclidean metric.
Center: fixed metric frozen at $M(\bar{x})$.
Right: state-dependent metric with region-based eigenvector rotation.
All trajectories converge to the same interior solution $\bar{x}=(0.4,-0.3)$,
but the transient geometry differs significantly across metrics.}
\label{fig:nl_int_traj}
\end{figure}

Figure~\ref{fig:nl_int_traj} shows representative trajectories from multiple
deterministic initial points.
Although all three metrics converge to the same interior solution,
their transient behavior differs markedly.
The Euclidean flow exhibits sharper directional changes, while the fixed metric
induces a global anisotropic bias.
In contrast, the state-dependent metric produces smoothly curved trajectories,
reflecting continuous rotation of the local eigenframe along the flow.
Since the solution is interior, these differences are purely geometric and are
not influenced by active constraints.

\begin{figure}
    \centering
    \includegraphics[width=0.75\linewidth]{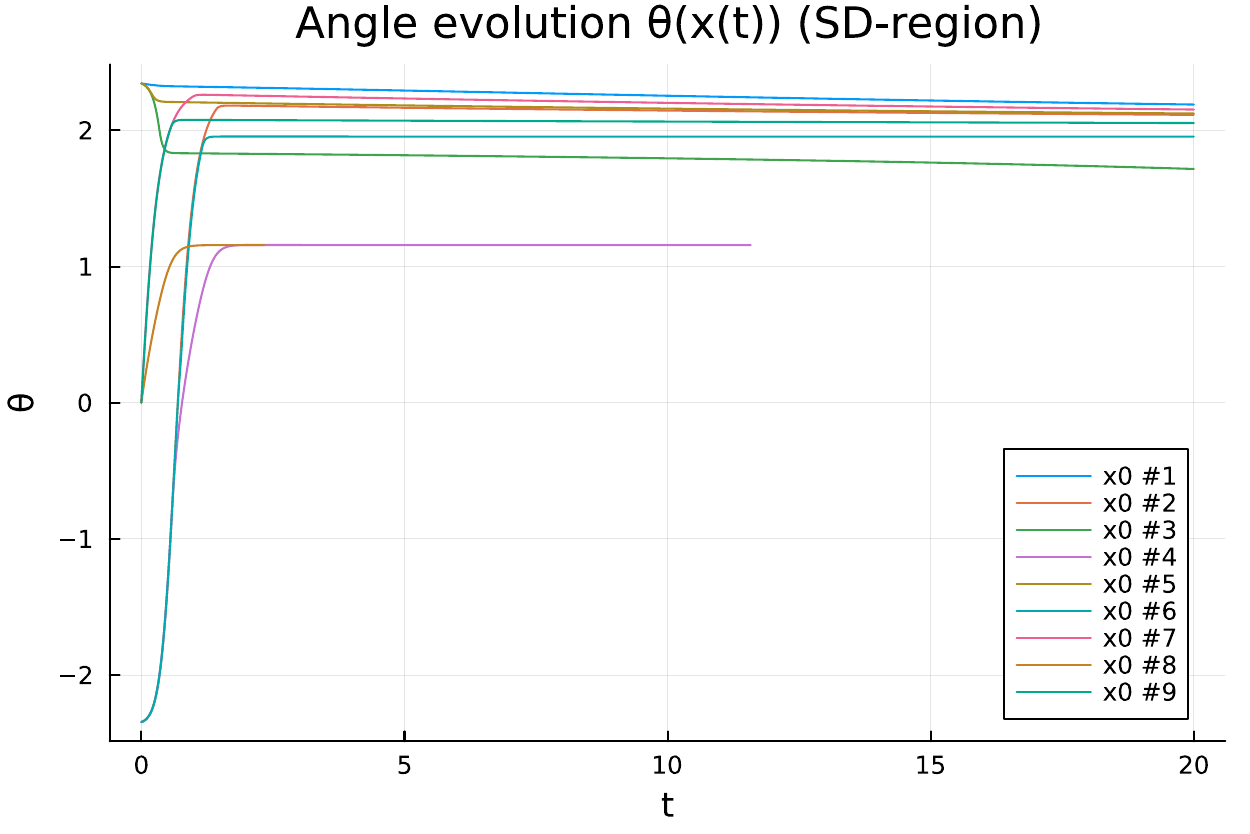}
    \caption{Evolution of the rotation angle $\theta(x(t))$ along trajectories
for the state-dependent metric in Example~3.
The angle varies during the transient phase and converges to
$\theta(\bar{x})$ as $x(t)\to \bar{x}$, illustrating that eigenframe rotation
is a transient geometric effect.}
\label{fig:nl_int_theta}
\end{figure}
The corresponding evolution of the rotation angle $\theta(x(t))$ for the
state-dependent metric is shown in Figure~\ref{fig:nl_int_theta}.
Trajectories start with different initial angles depending on their location,
but all converge to the common value $\theta(\bar{x})$ as $x(t)\to \bar{x}$.
This confirms that eigenframe rotation is a transient mechanism that shapes the
approach geometry, while asymptotic convergence remains governed by the local
linearization, in agreement with the theoretical results.

\medskip

To illustrate how the proposed state-dependent metric design extends beyond box constraints, Appendix~\ref{app:lp} presents an illustrative example based on a regularized linear program over a polyhedral feasible set.

\section{Related Work and Discussion}\label{sec:relatedwork}

Projection-based dynamical systems constitute a classical continuous-time approach
to solving variational inequalities $\mathrm{VI}(F,\mathcal S)$, where the system
evolves according to the residual between the current state and its projection
onto the feasible set $\mathcal S$. In their standard formulation, these dynamics
are defined with respect to a fixed Euclidean geometry and admit convergence
analysis via Lyapunov arguments under monotonicity-type assumptions on the
operator $F$
\cite{dupuis_dynamical_1993,nagurney_projected_1996,xia_general_2004,hu_solving_2006}.

The canonical model in this class is the projected dynamical system (PDS)
\[
\frac{dx}{dt} = \lambda\bigl(P_{\mathcal S}(x-\alpha F(x)) - x\bigr),
\]
where $P_{\mathcal S}$ denotes the Euclidean projection onto a closed convex set.
Introduced by Dupuis and Nagurney, this formulation provides a discontinuous
dynamical extension of variational inequalities, with trajectories constrained to
$\mathcal S$ through a Skorokhod-type reflection mechanism
\cite{dupuis_dynamical_1993}. Under continuity of $F$, Carathéodory solutions exist,
while uniqueness requires local Lipschitz continuity
\cite{nagurney_projected_1996, heemels_projected_2000}. Equilibria of this flow coincide exactly with solutions of $\mathrm{VI}(F,\mathcal S)$, since the fixed-point condition
$P_{\mathcal S}(x-\alpha F(x))=x$ is equivalent to the defining variational inequality
\cite{dupuis_dynamical_1993,nagurney_projected_1996}.

This projection--residual structure has been extensively studied and extended,
including formulations in Hilbert spaces, evolutionary variational inequalities,
and systems defined on prox-regular sets
\cite{friesz_dynamic_1996,cojocaru_projected_2005,hauswirth_projected_2020}. In all these settings,
however, the underlying metric structure is fixed a priori, even when generalized
solution concepts or nonsmooth analysis are required.

Projection neural networks (PNNs) realize the same dynamics in a neural or
hardware-oriented architecture,
\[
\frac{dx}{dt}= -x + P_{\mathcal S}(x-\alpha F(x)),
\]
and can be interpreted as exact continuous-time implementations (up to the choice of $\lambda$) of Euclidean
projected dynamical systems
\cite{xia_general_2004,hu_solving_2006}. Global convergence, Lyapunov stability,
and exponential rates under strong (pseudo)monotonicity have been established for
various classes of operators and feasible sets
\cite{xia_general_2004,hu_solving_2006,gao_novel_2005,hu_recurrent_2007}. These
results rely fundamentally on nonexpansiveness properties of the Euclidean
projection and on energy functions constructed within a fixed metric geometry.

A parallel line of research introduces \emph{scaled} or \emph{preconditioned}
projection methods, primarily in discrete time. Scaled gradient projection (SGP)
schemes replace the Euclidean projection with projections induced by
iteration-dependent positive definite matrices, leading to improved conditioning
and accelerated convergence for ill-conditioned problems
\cite{bonettini_scaled_2009}. From the discrete-time perspective, SD-SPNNs are closely related in spirit to variable-metric projected gradient methods, including scaled gradient projection schemes and quasi-Newton-type projection algorithms. In those methods, one selects a sequence of positive definite matrices $\{M_k\}$ and iterates
\[
x_{k+1} = P_{S,M_k^{-1}}\bigl(x_k - \alpha_k M_k F(x_k)\bigr),
\]
with $M_k$ usually adapted from curvature information or problem structure. The present framework can be viewed as a continuous-time analogue of such variable-metric iterations, where the discrete sequence $\{M_k\}$ is replaced by a smoothly varying map $M(x)$ along trajectories. The analysis in Sections~3–6 shows that, under uniform boundedness and Lipschitz regularity of $M(\cdot)$, one can retain well-posedness and Lyapunov-based stability results that mirror classical guarantees for fixed-metric projected dynamics, while explicitly accounting for the additional coupling introduced by metric variation. This highlights that many ideas from scaled projection methods have a natural continuous-time counterpart when formulated at the level of the underlying dynamical system. 

Related ideas appear in splitting and resolvent-based
algorithms for variational inequalities, where implicit metrics arise through
regularization or resolvent operators
\cite{noor_modified_2001,noor_implicit_2003}. While effective algorithmically,
these approaches modify the geometry only at the iteration level and do not yield
a continuously evolving metric within a single dynamical system.

Continuous-time dynamical systems incorporating non-Euclidean geometry have also
been studied via monotone operator theory and implicit differential inclusions,
notably through resolvent or proximal dynamics
\cite{csetnek_continuous_2020}. In these formulations, however, the induced metric
is fixed once the operator is specified. More recent flows emphasizing safety,
invariance, or robustness similarly retain Euclidean projections or fixed inner
products
\cite{allibhoy_anytime_2025}. Even when posed on manifolds or irregular domains,
the metric structure is assumed fixed throughout the evolution
\cite{hauswirth_projected_2020}.

A complementary line of work studies optimization and variational problems on Riemannian manifolds, where the metric is intrinsic to the geometry of the constraint set rather than externally designed. In Riemannian gradient and trust-region methods, the metric is typically fixed by the manifold structure, and algorithmic flexibility comes from choosing retractions or approximate second-order models, see, e.g. \cite{AbsilMahonySepulchre2008, ring_optimization_2012}. In contrast, the SD-SPNN framework treats the metric as an explicit design degree of freedom in an ambient Euclidean space: the feasible set $S$ remains a convex subset of $\mathbb{R}^n$, while the inner product $\langle \cdot,\cdot \rangle_{M(x)}$ is allowed to vary smoothly with the state. This leads to a geometry that is \emph{state-dependent but not intrinsic}: different choices of $M(\cdot)$ induce different flows on the same set $S$, even though the underlying manifold structure is unchanged. From this viewpoint, SD-SPNNs can be interpreted as a bridge between Euclidean projection neural network and more general Riemannian or preconditioned flows, retaining the simplicity of convex constraints in $\mathbb{R}^n$ while importing some of the flexibility of metric design usually associated with manifold optimization.

From this perspective, the state-dependent metric dynamics proposed in
Sections~2--4 occupy a distinct position. Classical projected dynamical systems
and projection neural networks correspond to the special case
$\matM(x)\equiv I$, while constant-metric instances $\matM(x)\equiv M\succ 0$ may
be interpreted as continuous-time analogues of scaled projection methods.
The key novelty lies in allowing the metric $\matM(x)$ to vary \emph{smoothly
along trajectories}. This introduces continuous-time geometric preconditioning
directly at the dynamical-systems level, while preserving exact equivalence
between equilibria and solutions of $\mathrm{VI}(F,\mathcal S)$ and retaining
stability guarantees under the same monotonicity assumptions on $F$.

The numerical experiments reported in this paper reflect this hierarchy.
Constant metric choices already yield significant improvements over the Euclidean
baseline in ill-conditioned problems, consistent with observations from discrete
scaled projection methods. State-dependent metrics provide additional benefits,
particularly in transient behavior and near active constraints, that cannot be
replicated by fixed scaling alone. At the same time, the effectiveness of metric
adaptation is problem dependent, underscoring that metric design is an intrinsic
component of projection neural network rather than a mere implementation detail.

Additional developments include fixed-time and predefined-time neurodynamic solvers, resolvent- and regularization-based dynamical systems for generalized and inverse variational inequalities, as well as data-driven and operator-learning formulations of VIs; see, e.g., \cite{jiang_efficient_2012,combettes_deep_2020,combettes_variational_2022,yang_two_2025,zheng_novel_2024,tran_globally_2025,anh_regularized_2025,xu_notes_2021,zhao_learning_2025}.
\medskip

From a neural network perspective, the SD-SPNN formulation inherits the architecture of classical projection neural networks, where the state dynamics are realized by interconnected units implementing linear combinations, nonlinear activations, and projection blocks. In that setting, the terminology “neural network” refers less to training from data and more to the analog or neuromorphic realization of dynamical systems that solve optimization or equilibrium problems in continuous time. The present work remains within this tradition: the state-dependent metric can be interpreted as a continuously updated weight or gain structure in the underlying circuit, modifying how the residual $x - P_{S,M^{-1}(x)}(x - \alpha M(x)F(x))$ is formed and fed back. At the same time, all convergence and stability results are derived at the dynamical-systems level, without invoking learning or data-driven adaptation. This places SD-SPNNs conceptually between classical neurodynamic solvers and purely geometric projection flows: they admit a neural implementation, but their analysis and design are governed by variational inequality structure and metric geometry rather than by training objectives.

Taken together, the preceding discussion situates the proposed state-dependent
metric projection neural network within the broader landscape of projection-based
methods for variational inequalities, clarifying both their conceptual lineage
and their point of departure from existing approaches. In contrast to classical
fixed-geometry flows and discrete preconditioning schemes, the present framework
embeds adaptive geometry directly into the continuous-time dynamics, introducing
analytical features that are not captured by existing models. We now turn to a
discussion of the implications, limitations, and future directions suggested by
this formulation.

\section{Impact and Future Research Directions} \label{sec:futurework}

The state-dependent scaled projection neural network (SD-SPNN) framework developed in this paper provides a flexible and analytically grounded setting for studying projection neural network for variational inequalities beyond fixed Euclidean geometries. Rather than introducing a single algorithmic variant, the framework treats the metric underlying the projection operator as an explicit design variable that may evolve smoothly with the state. Within this setting, Euclidean projection neural network, constant-metric scaling, and state-dependent preconditioning all appear as special cases of a single continuous-time model.

From a conceptual standpoint, this perspective departs from classical projection neural networks and scaled projection methods, where the metric is either fixed once and for all or modified externally at the iteration level. In the SD-SPNN formulation, metric design is embedded directly into the dynamical system itself. This clarifies the role of geometry in projection-based solvers for variational inequalities and highlights that the choice of metric is not merely a numerical detail but a structural component of the dynamics that interacts with stability and convergence properties.

Although the analysis here focuses on continuous-time dynamics, the SD-SPNN formulation naturally suggests new discrete-time algorithms. A first direction is to derive variable-metric projection methods that approximate the flow
\[
\frac{d x}{dt} = \lambda\bigl(P_{S,M^{-1}(x)}(x - \alpha M(x)F(x)) - x\bigr)
\]
through explicit or semi-implicit time discretizations, leading to schemes of the form
\[
x_{k+1} = P_{S,M_k^{-1}}\bigl(x_k - \alpha_k M_k F(x_k)\bigr), \qquad M_k \approx M(x_k).
\]
Such methods would connect SD-SPNNs to scaled gradient projection and quasi-Newton projection algorithms, while providing a continuous-time interpretation of their preconditioning behavior. A second direction is to develop splitting and forward–backward schemes with evolving preconditioners, in which the metric is updated along the iterates but remains regular enough for the continuous-time Lyapunov analysis to inform discrete-time convergence guarantees. Establishing rigorous links between the stability arguments used here and discrete-time convergence results in the presence of state-dependent metrics remains an open challenge.

The stability and convergence results obtained in this work rely on monotonicity or strong monotonicity assumptions on the operator $F$. Many applications of variational inequalities, however, involve weaker conditions such as pseudomonotonicity, quasimonotonicity, or structured nonmonotonicity. Extending the SD-SPNN framework to these settings raises nontrivial analytical questions, including how metric variation interacts with weaker notions of monotonicity and whether alternatives to strict Lyapunov decay—such as invariance principles, attractor-based arguments, or generalized dissipativity—can be used to obtain meaningful convergence statements. Related questions arise for generalized and inverse variational inequalities, equilibrium problems, and models with uncertainty or noise, where the interaction between state-dependent metrics and problem structure may be particularly subtle.

Another important direction concerns the systematic design of the state-dependent metric $M(\cdot)$. In the present analysis, $M(\cdot)$ is assumed to be prescribed and to satisfy uniform boundedness and smoothness conditions, but is otherwise arbitrary. This leaves considerable freedom in choosing metrics that reflect problem structure, local conditioning, or trajectory information. Future work may investigate adaptive or data-driven strategies for metric selection, for instance by coupling the SD-SPNN dynamics with estimators of curvature or sensitivity, or by restricting $M(x)$ to structured classes (diagonal, low-rank, block-diagonal) that balance geometric expressiveness with computational tractability. Bridging geometry-aware projection neural network with adaptive optimization and operator-learning methods, while maintaining analytical control over the resulting coupled system, is a promising but challenging avenue.

The present framework also has limitations. The smoothness and uniform boundedness assumptions on the metric, while natural for Lyapunov-based analysis, may be restrictive in applications where nonsmooth or abruptly changing preconditioners are desirable. Furthermore, scalability to very high-dimensional problems and the computational cost associated with evaluating $M(x)$ and computing metric projections are important practical considerations. These costs depend heavily on the structure of $M(x)$ and on the tractability of the resulting projection subproblems. Addressing these computational aspects, for example through structure-exploiting metrics, approximate projections, or inexact dynamics, is an essential step toward large-scale applications of state-dependent metric projection methods for variational inequalities.

\section*{Funding}
This research received no external funding.

\section*{Data and Code Availability}
No external data were used in this study. All numerical results were generated by the algorithms described in the paper. The code used to produce the numerical experiments is available from the corresponding author upon reasonable request.

\section*{Acknowledgements}

The author acknowledges the support of King Fahd University of Petroleum \& Minerals (KFUPM) and the Interdisciplinary Research Center for Smart Mobility and Logistics at KFUPM. The author also thanks Professor Qamrul Hasan Ansari for his valuable feedback and insightful input during the development of this work.

\printbibliography

\appendix

\section{An Illustrative Design Example: Regularized Linear Programs on Polyhedra}
\label{app:lp}

This appendix provides an illustrative design example demonstrating how
state-dependent metrics can be constructed and interpreted in relation to
polyhedral feasible sets.
The purpose of this example is not to propose an algorithm for linear
programming, but rather to show how the proposed state-dependent metric
projection framework naturally extends to polyhedral constraints and how
metric design influences transient geometry under projection-based dynamics.

\subsection{Regularized linear programs as strongly monotone variational inequalities}

Consider a linear program over a nonempty closed polyhedron
\[
\mathcal P := \{x \in \mathbb R^n : A_p x \le b_p\},
\]
with cost vector $c \in \mathbb R^n$.
To obtain a well-posed equilibrium problem with a unique solution, we consider
the Tikhonov-regularized formulation
\begin{equation}
\label{eq:reg_lp}
\min_{x \in \mathcal P} \; c^\top x + \frac{\mu}{2}\|x\|^2,
\qquad \mu > 0.
\end{equation}
The regularization is introduced solely to obtain a strongly monotone problem
with a unique equilibrium and does not alter the nature of the underlying linear
constraints. Problem~\eqref{eq:reg_lp} is equivalent to the variational inequality
$\mathrm{VI}(F_\mu,\mathcal P)$ with affine operator
\[
F_\mu(x) := c + \mu x,
\]
which is strongly monotone with constant $\mu$.
As a result, the variational inequality admits a unique solution
$x^\star \in \mathcal P$, and the continuous-time and discrete-time
projection-based dynamics studied in the main text apply directly.

\subsection{Proximal interpretation and projection geometry}

The solution of the regularized problem~\eqref{eq:reg_lp} admits a simple and
useful geometric interpretation.
Indeed, $x^\star$ can be written as
\begin{equation}
\label{eq:prox_proj}
x^\star
= \operatorname*{argmin}_{x \in \mathcal P}
\left\{ \frac{\mu}{2}\|x + \mu^{-1}c\|^2 \right\}
= \operatorname{Proj}_{\mathcal P}\!\left(-\mu^{-1}c\right),
\end{equation}
that is, $x^\star$ is the Euclidean projection of $-\mu^{-1}c$ onto the feasible
polyhedron.
Equivalently, $x^\star$ is the Euclidean proximal point of the indicator function
$\iota_{\mathcal P}$ evaluated at $-\mu^{-1}c$. This interpretation highlights that the equilibrium selected by the regularization is entirely determined by the geometry of $\mathcal P$. In contrast, the projection-based dynamics governing the approach to $x^\star$ are not unique and may depend on the choice of metric, even though the limiting
solution remains unchanged.

\subsection{State-dependent metric design for polyhedral constraints}

Within the proposed framework, the dynamics evolve according to a relaxed
metric projection of the form
\[
\frac{d x}{dt} = \lambda\bigl(
P_{\mathcal P}^{M(x)^{-1}}\bigl(x - \alpha M(x)F_\mu(x)\bigr) - x
\bigr),
\]
where $M(x)$ is a symmetric positive definite, state-dependent metric.
While the solution $x^\star$ of the regularized problem is fixed by
\eqref{eq:prox_proj}, the choice of $M(x)$ governs how trajectories interact
with the faces of the polyhedron during the transient phase. In the polyhedral setting, the geometry experienced by the trajectory changes as
different constraints become nearly active. This motivates metric designs in which the eigenvalues remain bounded and fixed, while the eigenvectors rotate smoothly as a function of the state. Such metrics adapt to the local geometry of $\mathcal P$ without encoding objective-specific information, thereby isolating geometric effects from
conditioning.  From this perspective, state-dependent metrics provide a mechanism for shaping boundary interaction and directional changes induced by projection, rather than
for accelerating convergence or altering the equilibrium.

\subsection{Illustrative numerical example}
\label{app:lp:numerics}

We consider the polyhedral feasible set $\mathcal P \subset \mathbb R^2$ defined by
\[
\mathcal P := \{x \in \mathbb R^2 : A_p x \le b_p\},
\qquad
A_p =
\begin{pmatrix}
1 & 0\\
0 & 1\\
-1 & 0\\
0 & -1\\
1 & 1\\
-1 & 2
\end{pmatrix},
\qquad
b_p =
\begin{pmatrix}
1\\
1\\
0.2\\
0.5\\
1.2\\
1.6
\end{pmatrix}.
\]
Thus $\mathcal P$ can be viewed as a box-type region with two additional linear
inequalities, and in particular it enforces the lower bounds $x_1 \ge -0.2$ and
$x_2 \ge -0.5$.  We start from the linear program
\[
\min_{x\in\mathcal P} \; c^\top x,
\qquad
c=
\begin{pmatrix}
1\\
0.35
\end{pmatrix},
\]
and use the Tikhonov-regularized model
\begin{equation}
\label{eq:app_reg_lp_instance}
\min_{x\in\mathcal P}\; c^\top x + \frac{\mu}{2}\|x\|^2,
\qquad \mu=0.08,
\end{equation}
which yields the strongly monotone affine operator $F_\mu(x)=c+\mu x$.
For this instance, the unique solution is the Euclidean projection
\[
x^\star = \operatorname{Proj}_{\mathcal P}\!\left(-\mu^{-1}c\right),
\]
which is highlighted in the trajectory plots.  We integrate the relaxed metric-projection dynamics
\[
\frac{d x}{dt} = \lambda\bigl(
P_{\mathcal P}^{M(x)^{-1}}\bigl(x-\alpha M(x)F_\mu(x)\bigr)-x
\bigr),
\]
with $\lambda=1$ and $\alpha=0.1$, comparing:
(i) the Euclidean metric $M(x)=I$,
(ii) a fixed metric $M(x)\equiv M(x^\star)$, and
(iii) a state-dependent metric $M(x)$ whose eigenframe rotates smoothly with the state.
All configurations use identical parameters and differ only through the metric choice.

\medskip

In this numerical example, the role of the metric is to modulate how projection onto
the feasible polyhedron $\mathcal P=\{x:Ax\le b\}$ interacts with the regularized
operator $F_\mu(x)=c+\mu x$.
All metrics are chosen independently of the objective vector $c$ and depend only on
geometric considerations. The Euclidean metric $M(x)=I$ serves as a baseline and corresponds to the standard
orthogonal projection onto $\mathcal P$.
As a first non-Euclidean variant, we consider a fixed metric $M(x)\equiv M^\star$,
obtained by freezing a state-dependent metric at the equilibrium point $x^\star$.
This fixed metric introduces a global anisotropy but does not adapt along the
trajectory. The state-dependent metric $M(x)$ is constructed so that its eigenvalues are fixed, while its eigenvectors rotate smoothly as a function of the state.
Concretely, $M(x)$ is defined as
\[
M(x)=R(\theta(x))^\top D\,R(\theta(x)),
\]
where $D$ is a constant diagonal matrix with prescribed positive entries and
$R(\theta)$ is a planar rotation matrix.
The angle $\theta(x)$ is a smooth scalar function of the state, designed to vary
across the feasible region.
This induces a continuous rotation of the metric eigenframe as the trajectory
evolves. This construction allows the metric to adapt to the local geometry encountered near different faces of the polyhedron, without altering the conditioning of the system.
All three metrics—Euclidean, fixed, and state-dependent—share the same equilibrium
$x^\star$.
Their effect is therefore confined to the transient geometry of the trajectories,
illustrating how a state-dependent metric can influence projection behavior in a
linear program without changing the underlying solution.

Figure~\ref{fig:lp_traj_compare} shows that all trajectories converge to the same
equilibrium $x^\star$, while exhibiting distinct transient geometry.
In particular, the state-dependent metric alters how the flow interacts with
the faces of $\mathcal P$, producing visibly different boundary interaction and
direction changes compared to Euclidean and fixed-metric baselines.

\begin{figure}[t]
\centering
\includegraphics[width=\textwidth]{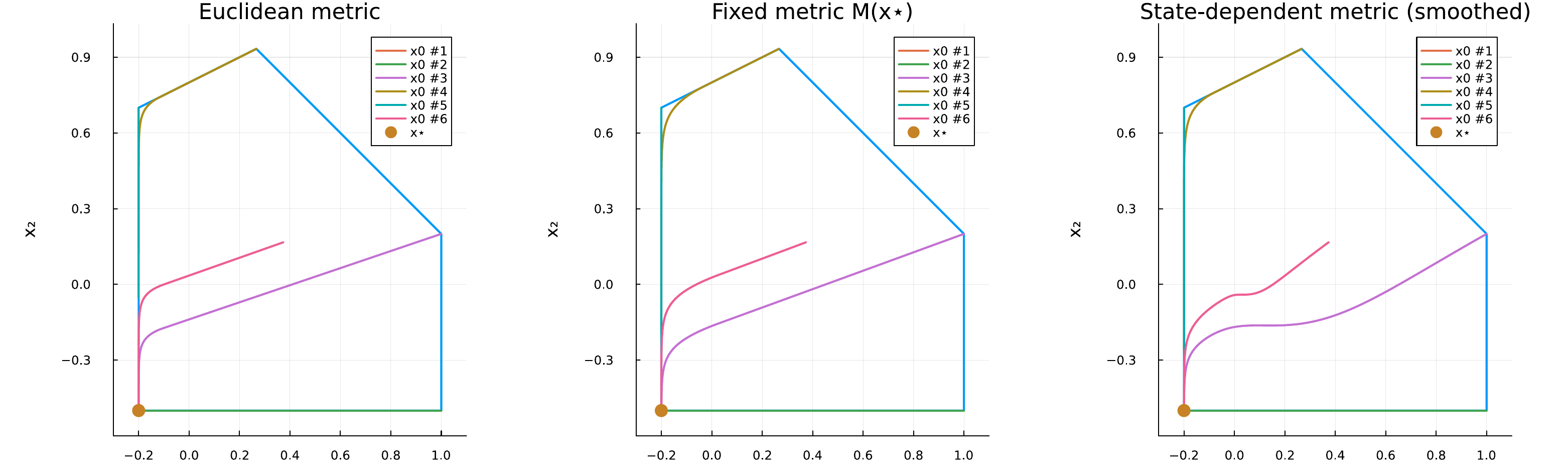}
\caption{Regularized LP on the polyhedron $\mathcal P$:
trajectory comparison for Euclidean, fixed, and state-dependent metrics
(with $\alpha=0.1$). The marked point denotes the unique equilibrium $x^\star$
of \eqref{eq:app_reg_lp_instance}.}
\label{fig:lp_traj_compare}
\end{figure}

For a representative initialization (not equal to $x^\star$), the evolution of the
residual norm and the Lyapunov function $V(x)=\tfrac12\|x-x^\star\|^2$ is reported in
Figure~\ref{fig:lp_residual_lyapunov}.
These diagnostics confirm consistency with the expected behavior: the residual decreases
along the trajectories, and $V(x(t))$ decays as the state approaches $x^\star$.

\begin{figure}[t]
\centering
\includegraphics[width=\textwidth]{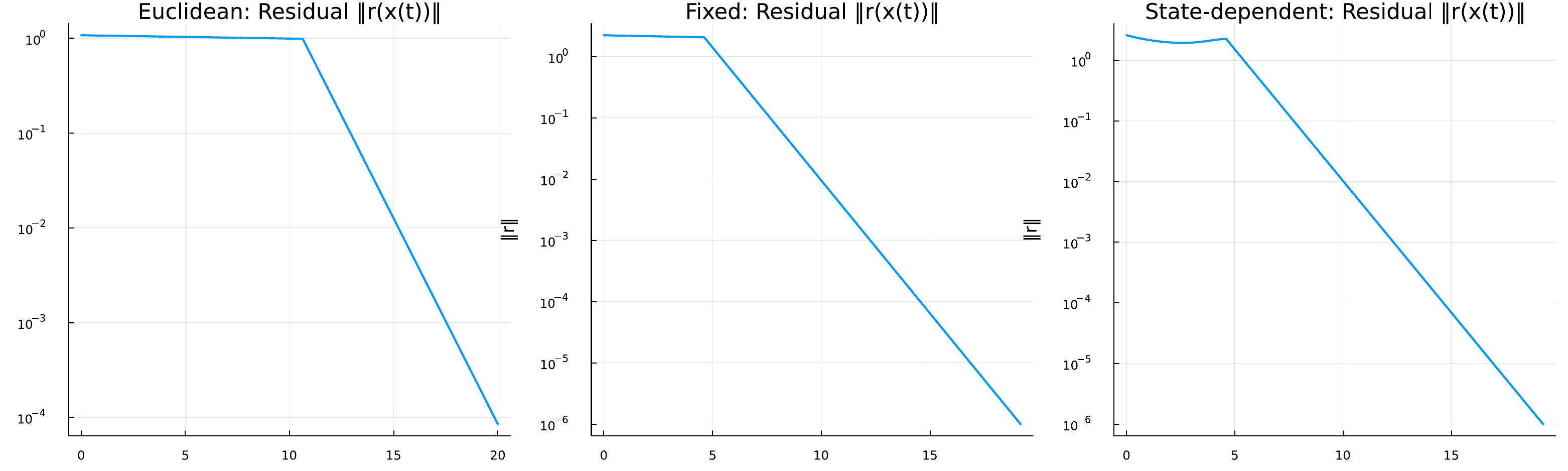}\\[1ex]
\includegraphics[width=\textwidth]{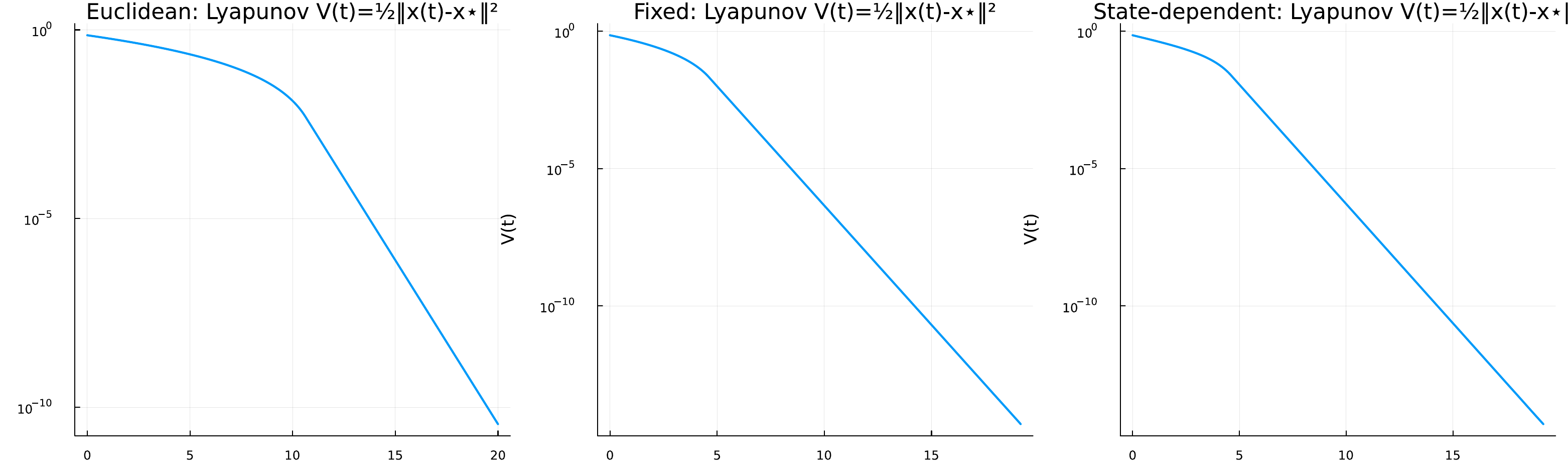}
\caption{Regularized LP on the polyhedron $\mathcal P$ (ODE mode, $\alpha=0.1$):
(residual) $\|r(x(t))\|$ and (Lyapunov) $V(x(t))=\tfrac12\|x(t)-x^\star\|^2$
for a representative initialization across Euclidean, fixed, and state-dependent metrics.}
\label{fig:lp_residual_lyapunov}
\end{figure}

\end{document}